\documentclass[reqno]{amsart}
\usepackage{amssymb, latexsym}

\allowdisplaybreaks

\newtheorem{thm}{Theorem}[section]
\newtheorem{cor}[thm]{Corollary}
\newtheorem{lem}[thm]{Lemma}
\newtheorem{prop}[thm]{Proposition}

\newtheorem{rem}[thm]{Remark}
\numberwithin{equation}{section}

\newcommand{\norm}[1]{\left\Vert#1\right\Vert}
\newcommand{\Norm}[1]{\left\Vert#1\right\Vert_{\dot{H}^1}}

\newcommand{\abs}[1]{\left\vert#1\right\vert}

\newcommand{\ab}{\alpha,\beta}
\newcommand{\brc}[1]{\left(   #1 \right) }
\newcommand{\set}[1]{\left\{#1\right\}}
\newcommand{\Jap}[1]{\left\langle#1\right\rangle}

\newcommand{\Japnb}{{\left\langle \nabla \right\rangle}}

\newcommand{\tu}{\tilde{u}}

\newcommand{\ep}{\varepsilon}
\newcommand{\Kab}{K_{\al,\beta}}
\newcommand{\Kcab}{{K}^c_{\al,\beta}}
\newcommand{\Kabp}{\mathcal{K}^+_{\al,\beta}}
\newcommand{\Kabn}{\mathcal{K}^-_{\al,\beta}}

\newcommand{\vab}{\varphi_{\al,\beta}}
\newcommand{\Hab}{H_{\al,\beta}}
\newcommand{\Hcab}{H^c_{\al,\beta}}
\newcommand{\Lab}{\mathcal{L}_{\al,\beta}}

\newcommand{\mab}{m_{\al,\beta}}
\newcommand{\mcab}{m^c_{\al,\beta}}

\newcommand{\tnj}{t_n^j}
\newcommand{\tnjj}{t_n^{j'}}
\newcommand{\Tnj}{T_n^j}
\newcommand{\Tnjj}{T_n^{j'}}
\def\snj{{\sigma^j_n}}
\def\sij{{\sigma^j_\infty}}

\def\taunj{{\tau^j_n}}
\def\tij{{\tau^j_\infty}}
\def\hij{h^j_\infty}

\def\utkn{\tilde{u}^k_n}

\def\xnj{x^j_n}
\def\xnjj{x_n^{j'}}
\newcommand{\vn}{v_n}

\newcommand{\vpn}{\varphi_n}

\newcommand{\hnj}{h_n^j}
\newcommand{\Unj}{U_n^j}
\newcommand{\Uij}{U^j_\infty}

\newcommand{\onk}{\omega_n^k}
\def\vj{v^j}

\newcommand{\hnjj}{h_n^{j'}}

\newcommand{\unj}{u_n^j}

\newcommand{\ujn}{u_{(n)}^j}
\newcommand{\ujjn}{u_{(n)}^{j'}}

\newcommand{\vnj}{v_n^j}

\newcommand{\K}{\mathcal{K}}
\newcommand{\pa}{\partial}
\newcommand{\al}{\alpha}

\newcommand{\vp}{\varphi}

\newcommand{\lag}{\langle}
\newcommand{\rag}{\rangle}

\newcommand{\sdg}{Schr\"odinger~}

\newcommand{\R}{\mathbb{R}}
\newcommand{\N}{\mathbb{N}}

\newcommand{\ra}{\rightarrow}

\newcommand{\dH}{\dot{H}}

\newcommand{\lf}{\left}
\newcommand{\rg}{\right}
\newcommand{\les}{\lesssim}

\newcommand{\C}{{\mathbb C}}

\newcommand{\beq}{\begin{equation}}
\newcommand{\eeq}{\end{equation}}
\newcommand{\ben}{\begin{eqnarray}}

\newcommand{\bsp}{\begin{split}}
\newcommand{\esp}{\end{split}}
\newcommand{\een}{\end{eqnarray}}
\newcommand{\beno}{\begin{eqnarray*}}
\newcommand{\eeno}{\end{eqnarray*}}

\begin{document}

\title[4D NLS with combined terms]{On the 4D Nonlinear Schr\"odinger equation with combined terms under the energy threshold}%

\author[Miao]{Changxing Miao}
\address{\hskip-1.15em Changxing Miao:
\hfill\newline Institute of Applied Physics and Computational
Mathematics, \hfill\newline P. O. Box 8009,\ Beijing,\ China,\
100088,} \email{miao\_changxing@iapcm.ac.cn}

\author[Zhao]{Tengfei Zhao}
\address{\hskip-1.15em Tengfei Zhao:
 \hfill\newline The Graduate School of China Academy of Engineering Physics, \hfill\newline P. O.
Box 2101,\ Beijing,\ China,\ 100088, } \email{zhao\_tengfei@yeah.net}

\author[Zheng]{Jiqiang Zheng}
\address{\hskip-1.15em Jiqiang Zheng \hfill\newline Universit\'e Nice Sophia-Antipolis, \hfill\newline 06108 Nice Cedex 02, France}
\email{zhengjiqiang@gmail.com, zheng@unice.fr}



\maketitle
\begin{abstract}
In this paper, we consider the longtime dynamics  of the solutions to focusing energy-critical Schr\"odinger equation with a defocusing energy-subcritical perturbation term under a ground state energy threshold in four spatial dimension.  This extends the results in
Miao et al. (Commun Math Phys 318(3):767-808, 2013, The dynamics of the NLS with
the combined terms in five and higher dimensions. Some topics in harmonic analysis and
applications, advanced lectures in mathematics, ALM34, Higher Education Press, Beijing,
pp 265-298, 2015) to four dimension without radial assumption and the proof of scattering
is based on the interaction Morawetz estimates developed in Dodson (Global well-posedness
and scattering for the focusing, energy-critical nonlinear Schr\"{o}inger problem in dimension
$d=4$ for initial data below a ground state threshold, arXiv:1409.1950), the main ingredients
of which requires us to overcome the logarithmic failure in the double Duhamel argument in
four dimensions.
\end{abstract}

\section{Introduction}

  In this paper, we consider the nonlinear \sdg  equation  with combined terms in four spatial dimension, i.e.
\begin{equation}\label{001}
\left\{
  \begin{aligned}
  i\partial_t u+\Delta u&=-|u|^2u+ |u|^{\frac{4}{3}}u,\quad (t,x)\in\R\times\R^4 \\
  u(0,x)&=\vp(x)\in H^1(\R^4),
  \end{aligned}
\right. \tag{$\mathrm{CNLS}$}
\end{equation}
where $\Delta=\sum\limits_{1\leq j\leq 4}\pa^2_{x_j}$  is the Laplace operator.

If $u$ is a solution to \eqref{001},
$$M(u)=\int_{\R^4} \abs{u}^2 dx  \text{\; and \;}
 E(u)=\int_{\R^4} \left[\frac{1}{2}\abs{\nabla u}^2 - \frac{1}{4} \abs{u}^4 +\frac{3}{10} \abs{u}^{\frac{10}{3}}\right] dx $$
are often said to be  the mass and energy of $u$ respectively. Through a standard method, one can check that the mass  and energy of a smooth solution to \eqref{001} conserve in time. The nonlinear term $\abs{u}^2u$ is $\dot{H}^1(\R^4)$ critical and
 $\abs{u}^{4/3}u$ is $\dot{H}^{1/2}(\R^4) $ critical according to the standard scaling analysis.

In general, the energy critical NLS is given by
\begin{equation}\label{002}
\left\{
  \begin{aligned}
  i\partial_t u+\Delta u&=\mu|u|^{\frac{4}{d-2}}u,\quad (t,x)\in\R\times\R^d \\
  u(0,x)&=\vp(x)\in \dH^1(\R^d)
  \end{aligned}
\right. \tag{$\mathrm{NLS}$}
\end{equation}
where $\mu=\pm1$ and  $d\geq 3$.  For the local well-posedness of
\eqref{002} in $\dot{H}^1(\R^d)$ or $H^1(\R^d)$,
 we refer to \cite{Cazenave2003,CW-MM-1988,CW-NLA-1990}.

If $\mu=1$, we call the Cauchy problem \eqref{002} is defocusing.
There are many results in the literature considering the defocusing cases. In \cite{Bourgain-JAMS-1999},
Bourgain first developed a method of reduction of energy and proved that any solution to  defocusing \eqref{002}
with radial initial data is scattering  in spatial dimension three. Colliander. etc.\cite{CKSTT-Annals-2008}
 removed the radial assumption by exploiting  a interaction Morawetz estimate, and we refer \cite{K-V-A&P2012} for
  another proof, which is based on the long time Strichartz estimates. This result was extended
  in  Ryckman and Visan \cite{R-V-AJM-2007} for dimension four (for another proof, we refer to
   \cite{Visan-IMRN-2012})and Visan \cite{Visan-Duke-2007} for five and higher dimensions by using the Morawetz estimates.

For the focusing case $(\mu=-1)$, Kenig and Merle
\cite{Kenig-Merle-Invention2006} gave out a decomposition of the
region if the energy of the solutions under a ground state
threshold for radial solutions and they proved that in the
dimensions $d\in\{3,4,5\}$, one of the two regions for global
well-posedness and scattering and the other region for finite time blow-up. The corresponding
results in the cases of five and higher dimensions were proved by Killip and Visan in \cite{K-V-AJM2010}
without radial assumption by employing the double Duhamel formula trick. Dodson \cite{Dodson-2014}
obtained  global well-posedness and scattering  results in dimension four under a ground state threshold
by developing the  long time Strichartz estimates.  But the analog nonradial case of dimension three remains open up to now.

For scattering results of the defocusing energy-subcritical NLS,
we refer to \cite{CKSTT-CPAM-2004,Dodson-2014-algebraic,Dodson-AJM-2016, Nakanishi-JFA1999,Nakanishi-TM2001,TVZ-Duke2007}
and the reference therein.

There are also a quantity of results for the Cauchy problem \eqref{002} with a
energy-subcritical nonlinearity perturbation.
In \cite{TVZ-CPDE2007}, Tao, Visan, and Zhang proved
the scattering results in $H^1(\R^d)$ of defocusing  \eqref{002}  with a defocusing perturbation
$|u|^p u$  $(\frac{4}{d}\leq p<\frac{4}{d-2} \text { and }d\geq 3)$ or with focusing  $|u|^p u$
$(\frac{4}{d}\leq p<\frac{4}{d-2} \text { and }d\geq 3)$ and  the small mass condition.
The latter results was extended in Killip, Oh, Pocovnicu, and Visan \cite{KOPV-2014},
where the scattering results were obtained  for the solutions to the  defocusing \eqref{002} with a
focusing term $|u|^2u$ in three spatial dimension if their initial datum belong to a certain region given by rescaling.

Miao, Xu, and  Zhao \cite{MXZ-1} proved scattering and finite time blowup results
for the focusing \eqref{002} perturbed by a $\dot{H}^{1/2}$-critical
defocusing term  in spatial dimension three. More precisely, for every radial initial
data $u_0\in H^1(\R^3)$ below a ground state threshold, the corresponding solution
$u$ is globally well-posed and scattering, if its energy functional is nonnegative or blows up at finite time, if its energy functional is negative.
And in \cite{MXZ-2}, they extends the result to five and higher spatial dimension without the radial assumption.
We refer to \cite{AIKN-DIE-2011,AKIN-SM-2013,Nawa2015} for the  focusing energy critical NLS with some focusing perturbation terms.


%

In this article, we consider the longtime dynamics behavior of \eqref{001} below the energy threshold.
First, we consider the variational derivation of the energy. As in \cite{Nakanishi-2010}, we denote
$$\vp^\lambda_{\al,\beta}(x)=e^{\al\lambda}\vp(e^{-\beta\lambda} x),$$
  where $(\al,\beta,\lambda)\in \R^3$ and  $\vp:\R^4\ra \mathbb{C}$. We will restrict $(\al,\beta)$ in a region,
\begin{equation*}
 \Omega=\set{(\al,\beta):\al\geq0,\;\;5\al+6\beta\geq0,\;(\al,\beta)\neq(0,0)}.
\end{equation*}
And for any functional $F$ of $H^1(\R^4)$, we define the variation derivation  by
$$\mathcal{L}_{\al,\beta} F(\phi)=\left. \frac{d}{d\lambda}    \right\arrowvert _{\lambda=0} F( \phi^\lambda_{\al,\beta}).$$
For each $(\al,\beta) \in \Omega $, we define
$$K_{\al,\beta}(\phi)= \mathcal{L}_{\al,\beta} E(\phi)=  (\al+\beta)\int \left[|\nabla\phi(x)|^2  -|\phi(x)|^4\right] dx
+\Big(\al+\frac{6}{5}\beta\Big)\int |\phi(x)|^\frac{10}{3} dx ,$$
whenever $\phi$ is a function in $H^1(\R^4)$. Consider the
minimization problem
$$m_{\al,\beta}=\inf\{E(\vp):\vp\in H^1(\R^4)\setminus\{0\},\;\; \Kab(\vp)= 0 \}.$$
Since  the nonlinear term in \eqref{001} is $\dot{H}^1$-critical growth with the
$\dot{H}^1$-subcritical perturbation, we will use the modified
energy later
$$E_c(u)=\int_{\R^4}\Big(\frac{1}{2}\abs{\nabla u}^2 - \frac{1}{4} \abs{u}^4\Big)\;dx.$$

In this paper we will study the solutions to \eqref{001} which start from the following two regions below the minimum $m_{\ab}$,
\beq
\mathcal{K}^+_{\ab}=\set{\vp\in H^1(\R^4):  E(\vp)<m_{\ab},\;\Kab(\vp)\geq0 }
\eeq
and
\beq
\mathcal{K}^-_{\ab}=\set{\vp\in H^1(\R^4):  E(\vp)<m_{\ab},\;\Kab(\vp)<0 },
\eeq
where $(\al,\beta)\in \Omega.$%

Before stating  the main theorem, we give the definition of
scattering which will be used later. A global solution $u$ to
\eqref{001} with $u(0)=u_0\in H^1(\R^4)$ is scattering, if there
exist  $u_\pm \in H^1(\R^4)$ such that
\begin{equation*}
\lim_{t\ra \pm\infty} \norm{u(t)-e^{it\Delta} u_\pm}_{H^1(\R^4)}=0.
\end{equation*}
Now we state our main theorem, which describes the dynamics of
\eqref{001} under the threshold $m_{\ab}$.
\begin{thm}\label{thm:main} Let $(\al,\beta)\in\Omega$.
For each $\vp\in \K^+_{\ab}$, we have that if $u$ is the solution to
\eqref{001} with  $u(0,x)=\vp$, then $u$ is global well-posed and
scattering in $H^1(\R^4)$. On the other hand, if $\vp$ is a radial
function  in $\K^-_{\ab}$, then the solution $u$ with the initial
data $\vp$ will blow up at finite time.
\end{thm}

\begin{rem}
By using the same argument in this paper one can also obtain the corresponding results of  \eqref{001} if $|u|^{4/3}u$ is replaced by  some more general defocusing $|u|^pu(1<p<2)$ terms.
\end{rem}

To prove the main theorem, we need the following property of the minimal $m_{\al,\beta}$.

\begin{prop}\label{prop:minima}
For $(\al,\beta)\in \Omega$, we have $$m_{\ab}=E_c(W)
=\frac{1}{4}\norm{\nabla W}^2_{L^2(\R^4)},$$ where
$W(x)$  is the ground state of the
massless equation $-\Delta \vp= |\vp|^2\vp$,
given by $W(x)=\big(1+\frac{|x|^2}8\big)^{-1}$.
\end{prop}

We remark that this proposition implies that $m_{\al,\beta}$ is
independent of $(\ab)$. We also have the properties of
$\K^+_{\al,\beta}$ and $\K^-_{\al,\beta}$,
\begin{prop}\label{prop:K-independent}
The regions $\K^+_{\ab}$ and $\K^-_{\ab}$ do not depend on $(\al,\beta)$ if $(\al,\beta)\in \Omega.$
\end{prop}
Based on this property, we can denote $\K^+=\K^+_{\ab} $ and
$\K^-=\K^-_{\ab} $ for simplicity. We will also prove the energy trapping
property which  manifest another important
property of the regions $\K^+$ and $\K^-$.
\begin{prop}\label{prop:u-in-K}
 Let $u:~I\times\R^4\to\C$ be the
solution to \eqref{001} with initial data $u(0,x)=u_0(x)\in
H^1(\R^4)$. Then, we have
\begin{enumerate}
\item If $u_0\in\mathcal{K}^-$, then for each $t\in I$,
$u(t)\in\mathcal{K}^-$  and
\begin{align}\label{uniform:K:negative12re}
\|\nabla u(t)\|_{L_x^2}>\|\nabla W\|_{L_x^2}.
\end{align}
\item If $u_0\in\mathcal{K}^+$, then for each $t\in I$,
$u(t)\in\mathcal{K}^+$ and
\begin{align}\label{uniform:K:positive12re}
\|\nabla u(t)\|_{L_x^2}<\|\nabla W\|_{L_x^2}.
\end{align}
\end{enumerate}

\end{prop}
This and the energy conservation law shows that if $u$ is a solution to \eqref{001} with the initial data $u(0,x)\in\K^+$ or $\K^-$, the solution flow $\{u(t),t\in I\}$ will remain in the regions $\K^+$ or $\K^-$, where $I$ is the maximal lifespan of $u$. Hence this proposition brings about many conveniences to our proof.

To prove the blowup results, we use the same method as in \cite{Kenig-Merle-Invention2006,MXZ-1,MXZ-2}. Indeed,  the estimation of
the differentials of the localized virial identity helps us  preclude the global existence of the solution which starts from $\K^-$.

On the other hand, to prove the scattering results, by local
well-posedness theory in Section 2, it suffices to show that the
global scattering size of $u\in \K^+$ is bounded by certain
constant. To this end, we turn to a proof of contradiction. More
precisely, suppose the energy threshold $E_*$ is less than $m$, thus
there exists a sequence of solutions $u_n$ in $\K^+$, with the
property that $  E(u_n) \ra E_*  $  and $\norm{u_n}_{ST(\R)} \ra
\infty   $ as  $n \ra \infty$, where $ST(\R) $ is the scattering
norm we will define later.

By making use of  the linear and nonlinear profile decomposition and
the stability lemma given by \cite{AKIN-SM-2013}, we can obtain a
critical element $u_c(t,x)$. We will also prove a crucial compactness property of
the critical element dynamics, that is, $\{u_c(t),0\leq t<\infty\}$
is  precompact  in $\dot{H}^s(\R^4)$ after module  the translation
symmetry, for all $s\in (0,1]$.

Finally, in the last step, extinction of the critical element $u_c$,
we use the interaction Morawetz estimates to deduce that $u_c$
actually vanishes.  This is a contradiction to the the local
well-poesdness theory(which implies that the solution with small
initial data is global well-posed and scattering). It is worthwhile
to note that, since we consider the scattering problem in
$H^1(\R^4)$, the mass of $u_c$ conserves with time and remains
bounded, which makes this step more easier than the analog step in
\cite{Dodson-2014}. This is the main reason that why we do not use the
longtime Strichartz estimates for the critical element $u_c$.

This paper is organized as follows: In the first part of Section 2, we give the
notations in this paper and recall some basic harmonic analysis
tools and the local-wellposed theory of \eqref{001}.
We prove Proposition \ref{prop:minima}--\ref{prop:u-in-K} by introducing the variation
 method in the second part of Section 2.
Section 3 will prove the finite time blowup  for solutions in
$\K^-$. In Section 4, we will construct the linear and nonlinear
profile decomposition of the solutions sequence to \eqref{001}. In
Section 5, we will finish the proof of the main theorem.


\section{Basic estimates and variational method}

\subsection{Basic Tools}
First we will present some  harmonic analysis tools which will be
used later. We define the Fourier transform on $\mathbb{R}^4$ by
\begin{equation*}
\aligned \widehat{f}(\xi):= \tfrac{1}{4\pi^2}\int_{\mathbb{R}^4}e^{-
ix\cdot \xi}f(x)\,dx.
\endaligned
\end{equation*}
Based on this, for each $s\in \R$, we define the differentiation
operator $ \phi(\nabla)$ by $\widehat{\vp(\nabla) f}(\xi)=\vp(i\xi)
\hat{f}(\xi).$ Hence we can define the homogeneous Sobolev norms by
$\|u\|_{\dot{H}^s(\R^4)}=\||\nabla|^su\|_{L^2(\R^4)}$  and the
inhomogeneous Sobolev norms by
$\|u\|_{H^s(\R^4)}^2=\|u\|_{\dot{H}^s(\R^4)}^2+\|u\|_{L^2(\R^4)}^2$,
for $s\in\R$.

We will also use the following two lemmas, which deal with the
fractional derivatives.
\begin{lem}[Fractional product rule, \cite{Christ-W-JFA-1991}]
Let $s\in (0,1]$ and $f,g\in \mathcal{S}(\R^4)$, then we have
\beq\label{fp} \||\nabla|^s(fg)\|_{L^p(\R^4)}\les \||\nabla|^s
f\|_{L^{q_1}(\R^4)} \|g\|_{L^{r_1}(\R^4)} +\|| f\|_{L^{q_2}(\R^4)}
\|\nabla|^sg\|_{L^{r_2}(\R^4)}, \eeq where
$\frac{1}{p}=\frac{1}{q_1}+\frac{1}{r_1}$,
$\frac{1}{p}=\frac{1}{q_2}+\frac{1}{r_2}$ and $1\leq p\leq
q_1,q_2,r_1,r_2 < \infty.$

\end{lem}

\begin{lem}[Fratianal chain rule, \cite{Christ-W-JFA-1991}]\label{frc-chain}
Let $G\in C^1(\mathbb{C})$ and $s\in(0,1]$, then for any Schwartz
function $u$ we have \beq\label{fc} \||\nabla|^s
G(u)\|_{L^p(\R^4)}\les \|G'(u)\|_{L^q(\R^4)} \||\nabla|^s
u\|_{L^r(\R^4)}, \eeq where $\frac{1}{p}=\frac{1}{q}+\frac{1}{r}$
and $1\leq p\leq  q,r <\infty.$
\end{lem}

Now, we are ready to introduce the Strichartz estimates and give the
local well-posedness of \eqref{001}. First, we say a pair of
exponents $(q, r)$ is  Schr\"odinger $\dot{H}^s$-admissible in
dimension four if
$$\frac{2}{q}=4\left(\frac{1}{2}-\frac{1}{r}\right)-s$$ and $2\leq q,r\leq \infty$. And we denote the dual exponent $q'$ to
 $q\in (1,\infty)$ by $ \frac{1}{q'} +\frac{1}{q}=1$.

\begin{lem}[Strichartz estimates, \cite{GV-CMP-1992,KT-AJM-1998,Strichartz-Duke-1977}] Let $(q,r)$  and $(\widetilde{q},\widetilde{r})$
 be Schr\"odinger $L^2$-admissible pairs in dimension four. If $\vp\in L^2(\R^4)$ and $f\in L_t^{\widetilde{q}'}  L_x^{\widetilde{r}'}(\R\times\R^4)$,
then we have
\begin{align}
  \|e^{it\Delta} \vp(x)\|_{L_t^qL_x^r(\R\times\R^4)}\; \les &\|\vp\|_{L^2(\R^4)}.\label{st1} \\
  \Big\| \int_0^t e^{i(t-\tau) \Delta} f(\tau,x) d\tau \Big\|_{L_t^qL_x^r(\R\times\R^4)}\les &\|f\|_{L_t^{\widetilde{q}'}
   L_x^{\widetilde{r}'}(\R\times\R^4)}.\label{st2}
\end{align}
\end{lem}
If $I\subseteq \R$ is a interval, we define some time-spatial
Strichartz spaces by
\begin{align*}
 S(I)=&L^\infty_{t}(I; L^2(\R^4))\;\cap\; L^2_{t}(I; L^4(\R^4)), \\
  W_1(I)=&L^6_t(I; L^6(\R^4)),\;\;    \;\;\;   V_1(I) = L^6_t(I; L^\frac{12}{5}(\R^4)),\\
   W_2(I)=&L^4_t(I; L^4(\R^4)),\;\;     \;\;\;  V_2(I) = L^4_t(I; L^\frac{8}{3}(\R^4)),\\
  ST(I)=&W_1(I)\cap W_2(I), \;\; \;V_0(I)=L_t^{3}(I;L^{3}(\R^4)).
\end{align*}
By the definition of the Schr\"odinger admissible pairs, one can find that $W_1$ is $\dot{H}^{1}$-admissible, $W_2$ is $\dot{H}^\frac{1}{2}$-admissible and $V_0$, $V_1$, $ V_2$ is
$L^2$-admissible. On the other hand, standard arguments show that if a solution $u$ to \eqref{001} is global, with $\|u\|_{ST(\R)}<+\infty$, then it scatters. In view of this, we
define  $ \norm{u}_{ST(I)} $ as the scattering size of $u$ on time
interval $I$ if $u$ is  a solution to Cauchy problem \eqref{001}.

For  the sake of later use, by the H\"older inequality and Lemma \ref{frc-chain}, we give some
nonlinear estimates:
\begin{equation}\label{nl-3}
  \big\| |\nabla|^s (|u|^2 u)\big\|_{L^{3/2}_{t,x} (I\times \R^4) } \leq\big\||\nabla|^s u\big\|_{V_0(I)} \|u\|^{2}_{W_1(I)},\text{\;\;\; for } s=0,\;1
\end{equation}
and
\begin{equation}\label{nl-4}
  \big\| |\nabla|^s (|u|^\frac{4}{3}u)\big\|_{L^{3/2}_{t,x} (I\times \R^4) } \leq\big\||\nabla|^s u\big\|_{V_0(I)}
  \|u\|^{\frac{4}{3}}_{W_2(I)},\text{\;\;\; for } s=0,\;1.
\end{equation}

As a consequence of  the Strichartz estimates and the nonlinear estimates, one
can obtain the local theory of \eqref{001}.
\begin{thm}[Local well-posedness, \cite{AKIN-SM-2013}]
Let $u_0\in H^1(\R^4)$ and $I$ be a time interval with $0\in I$.  We
have:
\begin{itemize}
  \item[(1)] There exists $\delta=\delta(\|u_0\|_{H^1(\R^4)})>0$ such that:
If \beq \norm{\lag\nabla \rag e^{it\Delta} u_0}_{V_2(I)} \leq
\delta,\eeq then there exists a solution $u\in C(I;H^1(\R^4))$ to the
Cauchy problem \eqref{001} with the  following properties:
\begin{align}
u(0)=&u_0 ,\\
 \|\lag \nabla \rag u\|_{S(I)}\leq&
\|u_0\|_{H^1},\\
\|\lag \nabla \rag u\|_{V_2(I)} \leq& 2
\|\lag\nabla \rag e^{it\Delta} u_0\|_{V_2(I)}.
\end{align}

\end{itemize}
Furthermore, assume that $u\in C(I_{max};H^1(\R^4))$ is a solution
to  \eqref{001}, where $I_{max}$ is the maximal lifespan of $u$.
Note that $I_{max}$ must be open by $(1)$.
\begin{itemize}
  \item[(2)] The mass and energy  conservation laws hold true,   $\forall\, t,t_0\in I_{max}$,
\beq  M(u(t))=M(u(t_0)), \eeq \beq E(u(t))=E(u(t_0)).\eeq
  \item[(3)] Let $T_{max}=\sup I_{max}$. If $T_{max}<+\infty$, then
  \beq \|u\|_{ST([T,T_{max}))}=\infty\; \text{ for any } T \in I_{max}. \eeq
A similar result holds when $T_{min}=\inf I_{max}>-\infty$.
  \item[(4)] If $\|u\|_{ST(I_{max})}< \infty$, then $I_{max}=\R$ and there exist $\phi_\pm\in H^1(\R^4)$ such that
\beq \lim\limits_{t\rightarrow \pm\infty}
\|u(t)-e^{it\Delta}\phi_\pm\|_{H^1(\R^4)}=0. \eeq

\end{itemize}
\end{thm}
Similar to \cite[Proposition 5.6]{AKIN-SM-2013}, we have the Perturbation theory.

\begin{prop}[Perturbation theory,\cite{AKIN-SM-2013}]\label{prop:stability}
 Let $I$ be an interval, $u\in C(I;H^1(\R^4))$ be a solution to \eqref{001}
  and $\tu$ be a function in $ C(I;H^1(\R^4))$.
  Let $A>0$ and $t_1\in I$ such that
\beq \label{lpt-1} \norm{u}_{L^\infty(I;H^1)}  +
\norm{u(t_1)-\tu(t_1)}_{H^1}+\norm{\tu}_{ST(I)} \leq A.\eeq Then
there exists $\delta>0$ depending on $A$ such that if
 \begin{align}\label{lpt-2}
\norm{\Japnb \left((i\pa_t+\Delta)\tu+|\tu|^2\tu-|\tu|^\frac{4}{3}\tu    \right)}_{L^\frac32(I\times\R^4)}\leq& \delta,\\\label{lpt-3}
 \norm{\Japnb e^{i(t-t_1)\Delta} [u(t_1)-\tu(t_1)]}_{V_2(I)}\leq& \delta.
 \end{align}
Then we have $\norm{\Japnb u}_{S(I)}<\infty$.
\end{prop}

\subsection{Variational methods}

First we recall the energy of the solution $u$ to \eqref{001} \beq
E(u)(t)=\frac{1}{2}\int |\nabla u(t,x)|^2 dx -\frac{1}{4}\int
|u(t,x)|^4 dx +\frac{3}{10}\int |u(t,x)|^\frac{10}{3} dx, \eeq
and the modified energy
\beq E_c(u)(t)=\frac{1}{2}\int |\nabla u(t,x)|^2
dx -\frac{1}{4}\int |u(t,x)|^4 dx .\eeq

Recall in the introduction, for any functional $F$ of $H^1$, we define its variation differential by
$$\mathcal{L}_{\al,\beta} F(\phi)=\left. \frac{d}{d\lambda}    \right\arrowvert_{\lambda=0} F( \phi^\lambda_{\al,\beta})=F\Big([(\al-\beta x\cdot\nabla) \phi](x)\Big),$$
where $\phi_{\al,\beta}^\lambda(x) =e^{\al\lambda}\phi(e^{-\beta\lambda}x)$, if $(\ab,\lambda)\in \R^3$.
Thus we have
$$K_{\al,\beta}(\phi):= \mathcal{L}_{\al,\beta} E(\phi)= (\al+\beta)\int \left[|\nabla\phi(x)|^2 - |\phi(x)|^4 \right]dx
+\Big(\al+\frac{6}{5}\beta\Big)\int |\phi(x)|^\frac{10}{3} dx .$$

By the definition of the region of $(\ab)$, \beq\label{Omega}
\Omega=\set{(\al,\beta):\al\geq0,\;\;5\al+6\beta\geq0,\;(\al,\beta)\neq(0,0)},
\eeq it is easy to check that $\al+\beta>0$  if $(\al,\beta)\in \Omega$.
\begin{prop}\label{prop:tend0}
For any $(\al,\beta)\in\Omega $, if $\set{\vp_n}_{n\geq1}$ is a sequence in
$H^1(\R^4)$ with
$$\lim\limits_{ n\rightarrow \infty} \norm{\vp_n}_{\dot{H}^1}=0,$$
then, we have\beq\label{positive K} K_{\al,\beta}(\vp_n)> 0 \eeq for
sufficient large $n$.
\end{prop}
\begin{proof}
From the Sobolev inequalities
$$ \|\vp_n\|_{L^4(\R^4)}\les \|\vp_n\|_{\dot{H}^1(\R^4)},$$
we have
$\|\vp_n\|^4_{L^4(\R^4)}=o(\|\vp_n\|^2_{\dot{H}^1(\R^4)})$ as $n\ra \infty$, which together with \eqref{Omega} implies \eqref{positive K}.
\end{proof}
We define $\overline{\mu}$  by
$$\overline{\mu}=\max\Big\{2(\al+\beta),\frac{10}{3}\al+4\beta\Big\}
=\begin{cases}2(\al+\beta)\quad \text{if}\quad 2\al\leq-3\beta,\\
\frac{10}{3}\al+4\beta\quad \text{if}\quad 2\al\geq-3\beta.
\end{cases} $$
By a direct calculation, we have
\begin{lem}\label{lem:dircp}
\beq \big(\overline{\mu}-\mathcal{L}_{\al,\beta}\big)E(\vp)=\begin{cases} \frac{\al+\beta}2\|\varphi\|_{L^4}^4-\frac{2\al+3\beta}{5}\|\vp\|_{L^\frac{10}{3}}^\frac{10}{3}\quad \text{if}\quad 2\al\leq-3\beta,\\
\Big(\frac{2}{3}\al+\beta\Big)\Norm{\vp}^2+\frac{1}{6}\al
      \norm{\vp}^4_{L^4}\quad \text{if}\quad 2\al\geq-3\beta.
\end{cases}
\eeq
 and
\begin{align}\nonumber
&\mathcal{L}_{\al,\beta}(\overline{\mu}-\mathcal{L}_{\al,\beta})E(\vp)\\
=&\begin{cases}
2(\al+\beta)^2\|\vp\|_{L^4}^4-\frac2{15}(2\al+3\beta)(5\al+6\beta)\|\vp\|_{L^\frac{10}{3}}^\frac{10}{3}\quad \text{if}\quad 2\al\leq-3\beta,\\
\Big(\frac{2}{3}\al+\beta\Big)(2\al+2\beta)\Norm{\vp}^2+\frac{2}{3}\al(\al+\beta)
\norm{\vp}^4_{L^4}\quad \text{if}\quad 2\al\geq-3\beta.
\end{cases}
\end{align}

\end{lem}

This lemma  impies that $H_{\al,\beta}(\vp)>0 $ and
$\mathcal{L}_{\al,\beta}H_{\al,\beta}(\vp)>0 $ for $\vp\in
H^1(\R^4)\setminus\{0\}$, where  $H_{\ab}$ is given by
\beq\label{H}
\begin{split}
H_{\al,\beta}(\vp)=&\left(1-\frac{\mathcal{L}_{\al,\beta}}{\overline{\mu}}\right)E(\vp)\\
                  =& \begin{cases}
                  \frac14\|\vp\|_{L^4}^4-\frac{2\al+3\beta}{10(\al+\beta)}
                  \|\vp\|_{L^\frac{10}{3}}^\frac{10}{3}
                  \quad \text{if}\quad 2\al\leq-3\beta,\\
                  \frac{2\al+3\beta}{10\al+12\beta}\Norm{\vp}^2+\frac{\al}{20\al+24\beta}\|\varphi\|_{L^4}^4\quad \text{if}\quad 2\al\geq-3\beta.
                  \end{cases}
\end{split}
\eeq

 As in the introduction, we define
$$m_{\al,\beta}=\inf \left\{E(\vp):\vp\in H^1(\R^4)\backslash \{0\},
K_{\al,\beta}(\vp)=0 \right\}.  $$

\begin{lem}\label{lem:equi11} For $(\al,\beta)\in\Omega$, we have
\beq\label{depend m}
\begin{split}
m_{\al,\beta}=&\inf\set{H_{\al,\beta}(\vp): \vp\in H^1(\R^4)\backslash \{0\},\;K_{\al,\beta}(\vp)\leq 0} \\
             =& \inf\set{H_{\al,\beta}(\vp): \vp\in H^1(\R^4)\backslash \{0\},\;K_{\al,\beta}(\vp)< 0}.
\end{split}
\eeq
\end{lem}
\begin{proof}
Let $m'_{\al,\beta}=\inf\set{H_{\al,\beta}(\vp): \vp\in H^1(\R^4)\backslash \{0\},\;K_{\al,\beta}(\vp)\leq 0}\;\;$ and
$$ m''_{\al,\beta}=\inf\set{H_{\al,\beta}(\vp): \vp\in H^1(\R^4)\backslash \{0\},\;K_{\al,\beta}(\vp)< 0}.$$
 If $K_{\al,\beta}(\vp)=0$, then we have $H_{\al,\beta}(\vp)=E(\vp)$.
Hence we have \beq\label{depend m 1} m_{\al,\beta}\geq \mab'. \eeq
On the other hand, if $\Kab(\vp)<0$, from Proposition
\ref{prop:tend0}, there exists $\lambda_0<0$ such that
$$\Kab(\vp^{\lambda_0}_{\al,\beta})=0.$$ By Lemma \ref{lem:dircp}, we have
$\mathcal{L_{\al,\beta}}\Hab(\phi)>0 \text{ for any } \phi\in H^1,$
which implies
$$\Hab(\vp^{\lambda_0}_{\al,\beta}) <\Hab(\vp).$$ This implies
$\mab\leq\mab'$ , which together with \eqref{depend m 1} implies
$\mab=\mab'$. For the second equality in \eqref{depend m}, one can
easily find that \beq\label{depend m 2} \mab'\leq\mab''. \eeq For
any $\vp\in H^1(\R^4)\backslash \{0\}$ such that $\Kab(\vp)\leq 0$,
from Lemma \ref{lem:dircp}, we have
$$\Lab \Kab(\vp)=\overline{\mu}\Kab(\vp)-\Hab(\vp)<0.$$ This implies for $\lambda>0$, $\Kab(\vp^\lambda_{\al,\beta})<0.$
And by definition of $\Hab$, we have as $\lambda\rightarrow 0,$
$$\Hab(\vp^\lambda_{\al,\beta})\rightarrow \Hab(\vp),  $$
which implies $\mab'\geq\mab''$.

\end{proof}

Next we will use the ($\dot{H}^1$-invariant) scaling argument to
remove the $L^\frac{10}{3}$-term (the lower regularity quantity than
$\dot{H}^1$) in $\Kab$, that is, to replace the constrained
condition $\Kab(\vp)<0$ by $\Kcab(\vp)<0$, where $$\Kcab(\vp)=
\mathcal{L}_{\al,\beta} E_c(\vp)= (\al+\beta)\int |\nabla\vp(x)|^2
dx -(\al+\beta)\int |\vp(x)|^4 dx.$$
And let
\begin{equation*}
\Hcab(\vp)=\Big(1-\frac{\Lab}{\bar{\mu}}\Big)E_c(\vp)= \begin{cases}
\frac14\|\vp\|_{L^4}^4
                  \quad \text{if}\quad 2\al\leq-3\beta,\\
                  \frac{2\al+3\beta}{10\al+12\beta}\Norm{\vp}^2+\frac{\al}{20\al+24\beta}\|\varphi\|_{L^4}^4\quad \text{if}\quad 2\al\geq-3\beta.\end{cases}
\end{equation*}

\begin{lem}\label{lem:equivl}
For $(\al,\beta)\in\Omega$, we have \beq\label{m=mc}
\begin{split}
m_{\al,\beta}=&\inf\set{H_{\al,\beta}(\vp): \vp\in H^1(\R^4)\backslash \{0\},\;K^c_{\al,\beta}(\vp)\leq 0} \\
             =& \inf\set{H_{\al,\beta}(\vp): \vp\in H^1(\R^4)\backslash \{0\},\;K^c_{\al,\beta}(\vp)<
             0}\\
             =& \inf\set{\Hcab(\vp): \vp\in H^1(\R^4)\backslash \{0\},\;K^c_{\al,\beta}(\vp)< 0}\\
             =& \inf\set{\Hcab(\vp): \vp\in H^1(\R^4)\backslash \{0\},\;K^c_{\al,\beta}(\vp)\leq 0}\\
             =&\inf\set{H_{\al,\beta}^c(\vp): \vp\in H^1(\R^4)\backslash
             \{0\},\;K^c_{\al,\beta}(\vp)=
             0}.
\end{split}
\eeq
\end{lem}
\begin{proof}
First, we denote
$ \mab^{(1)}$ and $\mab^{(2)}$ by  $$\mab^{(1)}=\inf\set{H_{\al,\beta}(\vp): \vp\in H^1(\R^4)\backslash \{0\},\;K^c_{\al,\beta}(\vp)\leq 0}$$
and
 $$\mab^{(2)}=\inf\set{H_{\al,\beta}(\vp): \vp\in H^1(\R^4)\backslash \{0\},\;K^c_{\al,\beta}(\vp)< 0}.$$
From the definition of $\Hab$ and the fact that $\Kab(\vp) \geq \Kcab(\vp)$, we have
\beq\label{depend mc 1}
\mab^{(2)}\leq\mab''=\mab.
\eeq
To prove $\mab^{(2)}\geq\mab''$, for any $\vp\in H^1(\R^4)\backslash \{0\}$ with $\Kcab(\vp)<0$, we have
$$
\Kab(\vp^\lambda_{1,-1})=(\al+\beta)\left(\norm{\nabla\vp}^2_{L^2}-\norm{\vp}^4_{L^4}
\right)
+e^{-\frac{2}{3}\lambda}\Big(\al+\frac{6}{5}\beta\Big)\norm{\vp}^{\frac{10}{3}}_{L^\frac{10}{3}}\rightarrow\Kcab(\vp),
$$
as $\lambda\rightarrow+\infty.$ And for $\lambda>0$,  $\Hab(\vp^\lambda_{1,-1})\leq \Hab(\vp),$ which implies $\mab^{(2)}\geq\mab''.$
On the other hand, it is trivial that
\beq\label{depend mc 2}
\mab^{(1)}\leq\mab^{(2)}.
\eeq
Next, we need to prove $\mab^{(2)}\leq\mab^{(1)}$. Let $\vp\in H^1(\R^4)\backslash \{0\}$ such that $\Kcab(\vp)\leq 0$, then  from
$$\Lab\Kcab(\vp)=2(\al+\beta)^2\left(  \norm{\nabla\vp}^2_{L^2} -2\norm{\vp}^4_{L^4}      \right)<0, $$
we have $\Kcab(\vp^{\lambda}_{\al,\beta})<0$, for $\lambda>0.$
As above we have $\Hab(\vp^{\lambda}_{\al,\beta})\rightarrow \Hab(\vp)$, as $\lambda\rightarrow0.$ This implies $\mab^{(2)}\leq\mab^{(1)}$.

Second, we denote  $$\mcab= \inf\set{\Hcab(\vp): \vp\in H^1(\R^4)\backslash \{0\},\;K^c_{\al,\beta}(\vp)= 0} .$$
One can easily find that
\beq\label{mc1}
\begin{split}
 &\inf\set{\Hcab(\vp): \vp\in H^1(\R^4)\backslash \{0\},\;K^c_{\al,\beta}(\vp)\leq 0} \\
 \leq & \inf\set{\Hcab(\vp): \vp\in H^1(\R^4)\backslash \{0\},\;K^c_{\al,\beta}(\vp)< 0}.
\end{split}
\eeq
Let $\vp\in H^1(\R^4)\backslash \{0\}$ such that $\Kcab(\vp)\leq 0$. For $\lambda >0$, we have $\Kcab(\vp^\lambda_{\al,\beta})<0$, and
$$\Hcab(\vp^\lambda_{\al,\beta})\rightarrow\Hcab(\vp) \text{, as\; } \lambda\rightarrow0.$$
 Hence, we have
\beq\label{mc2}
\begin{split}
 &\inf\set{\Hcab(\vp): \vp\in H^1(\R^4)\backslash \{0\},\;K^c_{\al,\beta}(\vp)\leq 0} \\
 \geq & \inf\set{\Hcab(\vp): \vp\in H^1(\R^4)\backslash \{0\},\;K^c_{\al,\beta}(\vp)< 0}.
\end{split}
\eeq
It is trivial that
\beq\label{mc3}
\inf\set{\Hcab(\vp): \vp\in H^1(\R^4)\backslash \{0\},\;K^c_{\al,\beta}(\vp)\leq 0}\leq\mcab.
\eeq
Let $\vp\in H^1(\R^4)\backslash \{0\}$ such that $\Kcab(\vp)< 0$. Then there exist  $\lambda_0< 0$ such that
\beq\label{mc4}
\Kcab(\vp^{\lambda_0}_{\al,\beta})= 0.
\eeq
From $\Lab\Hab(\phi)>0$ for any $\phi\in H^1(\R^4)\backslash \{0\}$, we have $\Hcab(\vp^{\lambda_0}_{\al,\beta})\leq \Hcab(\vp), $
which implies
\beq\label{mc5}
\inf\set{\Hcab(\vp): \vp\in H^1(\R^4)\backslash \{0\},\;K^c_{\al,\beta}(\vp)\leq 0}\geq\mcab.
\eeq
Combine    \eqref{mc1},  \eqref{mc2}, \eqref{mc3} \eqref{mc4} \eqref{mc5} together, we can obtain the last two inequalities in \eqref{m=mc}.

Finally, we just need to show $\mcab=\mab$. It is trivial that $\mcab\leq \mab.$ By the definition of $\Hcab $ and $\Hab$,
we just need to show $\mcab\geq \mab$ in the case of $2\al<-3\beta. $

Let $\phi\in H^1(\R^4)\backslash \{0\}$ such that $\Kcab(\vp)<0.$ Then for any $\lambda\in \R$, $\Kcab(\vp^\lambda_{1,-1})=\Kcab(\vp)$, and $$\Hab(\vp^\lambda_{1,-1})=\Hcab(\vp)-\frac{2\al+3\beta}{10(\al+\beta)}e^{-\frac{2}{3}\lambda} \norm{\vp}^\frac{10}{3}_{L^\frac{10}{3}} \rightarrow\Hcab(\vp) \text{ as } \lambda\rightarrow+\infty.$$
This implies $\mcab\geq \mab.$

 Therefore,
we complete the proof of Lemma \ref{lem:equivl}.
\end{proof}

We remark that by Lemma \ref{lem:equivl} and  the definition of
$\mcab$ and $\Hcab$, we have \beq\label{remark m}
\mab=\mcab=\inf\set{ E_c(\vp): \vp\in H^1(\R^4)\backslash \{0\},\;
\norm{\nabla\vp}^2_{L^2} =\norm{\vp}^4_{L^4} }. \eeq which implies
that $\mab$ is independent of $(\al,\beta)$ if $(\al,\beta)\in
\Omega$. Hence we can denote $m$ by $m=\mab$ for any $(\al,\beta)\in
\Omega$. Now, we can make use of the sharp Sobolev constant in
\cite{Aubin:Sharp contant:Sobolev, Talenti:best constant} to compute
the minimization $m$, which also shows Proposition
\ref{prop:minima}.
\begin{lem}\label{m=EcW}
For the minimization $m_{\al,\beta}$, we have
$$m=E_c(W)$$
\end{lem}
\begin{proof}
From Lemma \ref{lem:equivl} and \eqref{remark m}, we have \beq
\begin{split}
m= &\inf\set{E_c(\vp): \vp\in H^1(\R^4)\backslash \{0\},\;  \norm{\nabla\vp}^2_{L^2} =\norm{\vp}^4_{L^4} }\\
 = & \inf\set{\frac{1}{4} \norm{\nabla\vp}^2_{L^2}: \vp\in H^1(\R^4)\backslash \{0\},\;  \norm{\nabla\vp}^2_{L^2} =\norm{\vp}^4_{L^4} }\\
\geq & \inf\set{\frac{1}{4} \norm{\nabla\vp}^2_{L^2}: \vp\in H^1(\R^4)\backslash \{0\},\;  \norm{\nabla\vp}^2_{L^2} \leq\norm{\vp}^4_{L^4} },
\end{split}
\eeq
where the equality holds if and only if the minimization is attained by some $\vp$ with $\norm{\nabla\vp}^2_{L^2} =\norm{\vp}^4_{L^4}.$ Next, we have
\beq
\begin{split}
&\inf\set{\frac{1}{4} \norm{\nabla\vp}^2_{L^2}: \vp\in H^1(\R^4)\backslash \{0\},\;  \norm{\nabla\vp}^2_{L^2} \leq\norm{\vp}^4_{L^4} }\\
=&\inf\set{\frac{1}{4} \norm{\nabla\vp}^2_{L^2}: \vp\in \dot{H}^1(\R^4)\backslash \{0\},\;  \norm{\nabla\vp}^2_{L^2} \leq\norm{\vp}^4_{L^4} }\\
=& \inf\set{\frac{1}{4} \left( \frac{\norm{\nabla\vp}_{L^2}}{\norm{\vp}_{L^4}}\right)^4   : \vp\in \dot{H}^1(\R^4)\backslash \{0\},\;
 \norm{\nabla\vp}^2_{L^2} \leq\norm{\vp}^4_{L^4} }\\
=&\frac{1}{4}\|W\|_{L^4}^{4}\\
=&E_c(W),
\end{split}
\eeq
where we used the fact that $H^1$ is dense in $ \dot{H}^1$ and the sharp Sobolev inequality
$$ \|\vp\|_{L^4(\R^4)} \leq \|W\|^{-1}_{L^4(\R^4)} \|\nabla\vp\|_{L^2(\R^4)} .  $$
\end{proof}
Hence we can define $\mathcal{K}^+_{\al,\beta}$ and $\mathcal{K}^-_{\al,\beta}$  by
\beq
\mathcal{K}^+_{\al,\beta}=\set{\vp\in H^1(\R^4):  E(\vp)<m,\;\Kab(\vp)\geq0 }
\eeq
and
\beq
\mathcal{K}^-_{\al,\beta}=\set{\vp\in H^1(\R^4):  E(\vp)<m,\;\Kab(\vp)<0 }.
\eeq

Now we are going to prove that the regions $\mathcal{K}^+_{\al,\beta}$ and  $\mathcal{K}^-_{\al,\beta}$ are independent of $(\al,\beta)$.
The proof is mainly based on the minimal property of $m_{\al,\beta}$ and the relation between energy $E$ and its variational derivatives $\Kab$,
which is similar to Lemma 2.9 in \cite{Nakanishi-2010}.
\begin{lem}\label{independent K}
If $(\al,\beta) \in \Omega$, then $\mathcal {K}^{\pm}_{\al,\beta}$ are independent of $(\al,\beta)$.
\end{lem}
\begin{proof}

From the definition of $\Omega$ in \eqref{Omega}, we have  $\al+\beta>0$, if $(\al,\beta)\in \Omega$.
First, $\Kabp\cup\Kabn$ is independent of $(\al,\beta)$, since
\beq\label{K depend}
\Kabp\cup\Kabn=\set{\vp\in H^1(\R^4):  E(\vp)<m}. 
\eeq
For $\vp\in \Kabp$, by the definition of $m$, if $\Kab(\vp)=0$, then $\vp=0$. Hence if $\vp\in\Kabp$ and $\vp\neq0$, then we have $\Kab(\vp)>0$.
From this we have the scaling variation  of $\vp\in \Kabp$, $\vp^\lambda_{\al,\beta}$ such that the following properties:
\begin{enumerate}
  \item[(1)] By the definition of $m$, $\Kab(\vp^{\lambda}_{\al,\beta})>0$ is preserved, provided $E(\vp^\lambda_{\al,\beta})<m$;
  \item[(2)] $E(\vab^\lambda)$ does not increase, as $\lambda\leq 0$ decreases, if $$\Lab E(\vp^\lambda_{\al,\beta})=\Kab(\vp^\lambda_{\al,\beta})>0. $$
\end{enumerate}

If there exists $\vp\in \Kabp\cap\mathcal{K}^-_{\al',\beta'}$, where $(\al,\beta), (\al',\beta')\in \Omega$.
By the definition of $m$ and the fact that $m>0$,  we have $\vp\neq0$ and $\Kab(\vp)>0$.

Then, from the above two facts, we have $E(\vab^\lambda)<m$, for
$\lambda\leq0$. And by $\al+\beta>0$, we have
$\vab^\lambda\rightarrow0$ in $\dot{H}^1$, as
$\lambda\rightarrow-\infty$.  But by Proposition \ref{prop:tend0},
there exists $\lambda_0<0$ such that
$K_{\al',\beta'}(\vp^{\lambda_0}_{\al,\beta})=0$, which together
with $E(\vab^{\lambda_0})<m$ contradicts the minimization of $m$.
Hence $\Kabp\cap\mathcal{K}^-_{\al',\beta'}=\emptyset$ for any
$(\al,\beta), (\al',\beta')\in \Omega$,  which and \eqref{K depend}
imply the claim.

\end{proof}

After the computation of the minimization $m$, we next give some
variational estimates.

\begin{lem}\label{lem:kgeq0} For any $\varphi\in H^1(\R^4)$ with $K_{2,-1}(\varphi)\geq0$, we have
\beq \label{equ:2.8} \int_{\R^4}\Big(\frac14|\nabla
\varphi|^2+\frac1{10}|\varphi|^\frac{10}3\Big)\leq
E(\varphi)\leq\int_{\R^4}\Big(\frac12|\nabla
\varphi|^2+\frac{3}{10}|\varphi|^\frac{10}{3}\Big). \eeq

\end{lem}

\begin{proof}
On one hand, the right hand side of \eqref{equ:2.8} is trivial. On
the other hand, by the definition of $E$ and $\Kab$, we have
$$E(\varphi)=\int_{\R^4}\Big(\frac14|\nabla \varphi|^2+\frac1{10}|\varphi|^\frac{10}3\Big)+\frac14K_{2,-1}(\varphi),$$
which implies the left hand side of \eqref{equ:2.8}.
\end{proof}

At the last of this section, we give the uniform bounds on the
scaling derivative functional $\Kab(\varphi)$ with the energy
$E(\varphi)$ below the threshold $m$, which plays an important role
for the blow-up and scattering analysis.

\begin{lem}\label{uniform bound}
For any $\varphi \in H^1$ with $E(\varphi)<m$.
\begin{enumerate}
\item If $K_{2,-1}(\varphi)<0$, then
\begin{align}\label{uniform:K:negative}
K_{2,-1}(\varphi)  \leq -\frac83\big(m-E(\varphi)\big),
\end{align}
and
\begin{align}\label{uniform:K:negative12}
\|\nabla \varphi\|_{L_x^2}>\|\nabla W\|_{L_x^2}.
\end{align}
\item If $K_{2,-1}(\varphi)\geq 0$, then
\begin{align}\label{uniform:K:positive}
K_{2,-1}(\varphi)\geq \min\left(\frac83(m-E(\varphi)), \frac{1}{5}
\big\|\nabla \varphi \big\|^2_{L^2} +  \frac1{100}
\big\|\varphi\big\|^\frac{10}{3}_{L^\frac{10}{3}} \right),
\end{align}
and
\begin{align}\label{uniform:K:positive12}
\|\nabla \varphi\|_{L_x^2}<\|\nabla W\|_{L_x^2}.
\end{align}
\end{enumerate}
\end{lem}
\begin{proof} By Lemma \ref{lem:dircp}, for any $\varphi \in H^1$, we have
\begin{align*}
\mathcal{L}_{2,-1}^2 E(\varphi) = \frac83 \mathcal{L}_{2,-1}
E(\varphi)- \frac{2}{3}\big\|\nabla \varphi\big\|^2_{L^2} - \frac43
\big\|\varphi\big\|^4_{L^4}.
\end{align*}
Let $j(\lambda)=E(\varphi^{\lambda}_{2,-1})$, then we have
\begin{align}\label{diff J}
j''(\lambda) =   \frac83j'(\lambda) - \frac23
e^{2\lambda}\big\|\nabla \varphi\big\|^2_{L^2} -
\frac43e^{4\lambda}\big\|\varphi\big\|^4_{L^4}.
\end{align}

\noindent{\bf Case I:} If $K_{2,-1}(\varphi)<0$, then by
$$\lim_{\lambda\to-\infty}\|\varphi^\lambda_{2,-1}\|_{\dot H^1}^2=0,$$ Proposition \ref{prop:tend0} and the continuity of $K_{2,-1}$ in
$\lambda$, there exists a negative number $\lambda_0<0$ such that
$K_{2,-1}(\varphi^{\lambda_0}_{2,-1})=  0$, and
\begin{align*}
K_{2,-1}(\varphi^{\lambda}_{2,-1})<   0, \;\; \forall\; \;
\lambda\in (\lambda_0, 0).
\end{align*}
By the definition of $m$, we obtain
$j(\lambda_0)=E(\varphi^{\lambda_0}_{2,-1}) \geq m$. Now by
integrating \eqref{diff J} over $[\lambda_0, 0]$, we have
\begin{align*}
\int^0_{\lambda_0} j''(\lambda)\; d\lambda \leq
\frac83\int^0_{\lambda_0} j'(\lambda)\; d\lambda.
\end{align*}
This yields  that
\begin{align*}
K_{2,-1}(\varphi)=j'(0)-j'(\lambda_0)\leq
\frac83\left(j(0)-j(\lambda_0)\right)  \leq -\frac83
\big(m-E(\varphi)\big),
\end{align*}
which implies \eqref{uniform:K:negative}.

Since $K_{2,-1}(\varphi)<0$, we have by Lemma \ref{lem:equi11}  \beq
m \leq H_{2,-1}(\varphi)< \frac14 \big\|\varphi\big\|^4_{L^4}. \eeq
where we have used the fact that $K_{2,-1}(\varphi)<0$ in the second
inequality. By the fact $m=\frac14 \left(C^*_4\right)^{-4}$ and the
Sharp Sobolev inequality, we have \beq \big\|\nabla
\varphi\big\|^4_{L^2} \geq (C^*_4)^{-4} \big\|\varphi\big\|^4_{L^4}
> (4m)^2, \eeq which implies that $\big\|\nabla
\varphi\big\|^2_{L^2} > 4m=\|\nabla W\|_{L_x^2}^2$. Thus, we
conclude \eqref{uniform:K:negative12}.

\noindent{\bf Case II:} $K_{2,-1}(\varphi) \geq 0$. We divide it
into two subcases:

When $4 K_{2,-1}(\varphi)\geq  \big\|\varphi\big\|^4_{L^4}$. Since
\begin{align*}
\frac43 \int_{\R^4}\big|\varphi\big|^4 \; dx =
-\frac43K_{2,-1}(\varphi) + \int_{\R^4} \left(\frac43 \big| \nabla
\varphi \big|^2 + \frac1{15} \big|  \varphi \big|^\frac{10}{3}
\right) \; dx,
\end{align*}
then we have
\begin{align*}
\frac{16}3 K_{2,-1}(\varphi)\geq -\frac43K_{2,-1}(\varphi) +
\int_{\R^4} \left(\frac43 \big| \nabla \varphi \big|^2 + \frac1{15}
\big| \varphi \big|^\frac{10}{3} \right) \; dx,
\end{align*}
which implies that
\begin{align*}
 K_{2,-1}(\varphi)\geq \frac{1}{5} \big\|\nabla
\varphi \big\|^2_{L^2} +  \frac1{100}
\big\|\varphi\big\|^\frac{10}{3}_{L^\frac{10}{3}}.
\end{align*}

When $4 K_{2,-1}(\varphi) \leq  \big\|\varphi\big\|^4_{L^4}$. By
\eqref{diff J}, we have for $\lambda=0$
\begin{align}
0< &  \frac{16}3j'(\lambda) <
\frac43e^{4\lambda}\big\|\varphi\big\|^4_{L^4}, \nonumber\\
j''(\lambda) = \frac83j'(\lambda) -& \frac23
e^{2\lambda}\big\|\nabla \varphi\big\|^2_{L^2} -
\frac43e^{4\lambda}\big\|\varphi\big\|^4_{L^4} \leq
-\frac83j'(\lambda). \label{evolution j}
\end{align}
By the continuity of $j'$ and $j''$ in $\lambda$, we know that $j'$
is an accelerating decreasing function as $\lambda$ increases until
$j'(\lambda_0)=0$ for some finite number $\lambda_0>0$ and
\eqref{evolution j} holds on $[0, \lambda_0]$.

By $ K_{2,-1}(\varphi^{\lambda_0}_{2,-1})=j'(\lambda_0)=0, $ we know
that
\begin{align*}
E(\varphi^{\lambda_0}_{2,-1})\geq m.
\end{align*}
Now integrating \eqref{evolution j} over $[0, \lambda_0]$, we obtain
that
\begin{align*}
-K_{2,-1}(\varphi)=j'(\lambda_0)-j'(0) \leq -\frac83
\big(j(\lambda_0)-j(0)\big) \leq -\frac83 (m-E(\varphi)),
\end{align*}
which implies \eqref{uniform:K:positive}.

On the other hand, we have by \eqref{equ:2.8}
\begin{align*}
\frac14\int_{\R^4}|\nabla \varphi|^2\;dx\leq
E(\varphi)<m=\frac14\int_{\R^4}|\nabla W|^2\;dx,
\end{align*}
and so \eqref{uniform:K:positive12} follows.
\end{proof}

As a consequence of Lemma \ref{uniform bound}, energy conservation,
continuous argument and Lemma \ref{independent K}, we deduce the
following lemma.

\begin{lem}\label{lem:keylemma}
Let $(\al,\beta)\in\Omega.$ Let $u:~I\times\R^4\to\C$ be the
solution to \eqref{001} with initial data $u(0,x)=u_0(x)\in
H^1(\R^4)$. Then, we have
\begin{enumerate}
\item If $u_0\in\mathcal{K}^-_{\al,\beta}$, then for each $t\in I$,
$u(t)\in\mathcal{K}^-_{\al,\beta}$  and
\begin{align}\label{uniform:K:negative12re}
\|\nabla u(t)\|_{L_x^2}>\|\nabla W\|_{L_x^2}.
\end{align}
\item If $u_0\in\mathcal{K}^+_{\al,\beta}$, then for each $t\in I$,
$u(t)\in\mathcal{K}^+_{\al,\beta}$ and
\begin{align}\label{uniform:K:positive12re}
\|\nabla u(t)\|_{L_x^2}<\|\nabla W\|_{L_x^2}.
\end{align}
\end{enumerate}

\end{lem}

\begin{rem}\label{rem:positi}
By sharp Sobolev embedding:
$\|\varphi\|_{L_x^4}^4\leq\frac{\|\nabla\varphi\|_{L_x^2}^4}{\|\nabla
W\|_{L_x^2}^4}$, we know that if $\|\nabla\vp\|_{L_x^2}<\|\nabla
W\|_{L_x^2}$, then
$$E_c(\vp)=\frac12\int_{\R^4}|\nabla\vp|^2\;dx-\frac14\int_{\R^4}|\vp|^4\;dx>\frac14\int_{\R^4}|\nabla\vp|^2>0.$$
\end{rem}

\section{Blow up of $\mathcal{K}^-$}
In this section, we will prove  that the radial solution of \eqref{001} in
$\mathcal{K}^-$ blows up at finite time.

Let $\phi$ be a smooth, radial function satisfying $|\pa^2_r\phi(r)|\leq 2$, $\phi(r)=r^2 $ for $r\leq 1$, and $\phi(r)=0$ for $r\geq3$.
For $R\geq1 $, we define $$\phi_R(x)=R^2\phi(\frac{|x|}{R})\;\text{ and }\; V_R(x)=\int_{\R^4} \phi_R(x)|u(t,x)|^2dx.$$
Let $u(t,x)$ be a radial solution to \eqref{001}, then by a direct computation,  we have
\beq
\partial_t V_R(x)=2 {\rm Im} \int_{\R^4}[\overline{u}\pa_j u](t,x)\pa_j\phi_R(x) dx,
\eeq and
\begin{align*}
\pa^2_t V_R(x)= & 4{\rm Re} \int_{\R^4}\pa_j\pa_k\phi_R(x)\pa_k\overline{u}(t,x)\pa_j u(t,x) dx \\
                & -\int_{\R^4}\left[ \Delta^2\phi_R |u|^2+\Delta\phi_R(|u|^4-\frac{4}{5}|u|^\frac{10}{3})\right] dx\\
=& 4\int_{\R^4} \phi_R''|\nabla u|^2dx-\int_{\R^4}\left[
\Delta^2\phi_R |u(t,x)|^2
 +\Delta\phi_R(x)(|u|^4-\frac{4}{5}|u|^\frac{10}{3})(t,x)\right]
 dx\\
\leq & 4\int_{\R^4} \left[2|\nabla u|^2-2|u|^4 +\frac{8}{5} |u|^{\frac{10}{3}} \right](t,x)dx\\
                   &+\frac{C}{R^2} \int_{R\leq|x|\leq 3R} |u|^2dx+C\int_{R\leq|x|\leq 3R} \left[|u|^4+|u|^\frac{10}{3}
                   \right]dx.
\end{align*}
By the radial Sobolev inequality, we have
\begin{align*}
\|f\|_{L^\infty(|x|\geq
R)}\leq&\frac{c}{R^\frac32}\|f\|_{L^2_x(|x|\geq R)}^\frac12\|\nabla
f\|_{L^2_x(|x|\geq R)}^\frac12.
\end{align*}
Therefore, by mass conservation and Young's inequality, we know that
for any $\epsilon>0$ there exist sufficiently large $R$ such that
\begin{align*}
\pa_t^2V(t)\leq&8K_{2,-1}(u(t))+\epsilon \big\|\nabla
u(t,x)\big\|^2_{L^2 }
+ \epsilon^2\\
=&32E(u(t))-(8-\epsilon)\|\nabla u(t,x)\big\|^2_{L^2
}-\frac{16}5\|u\|_{L_x^\frac{10}3}^\frac{10}3+ \epsilon^2.
\end{align*}
This together with Lemma \ref{lem:keylemma} and $E(u)<(1-\delta) m$
for some $\delta>0$ implies  that \beq \pa_t^2V_R(t)\leq -32 \delta
m+ C\ep\|\nabla W\|^2_{L^2}+C\ep^2. \eeq Finally, if we choose $\ep$
sufficient small, we can obtain $\pa_t^2V_R(t)\leq -16\delta_2 m$,
which implies that $u$ blows up in finite time.

\section{Profile decomposition}
In this section, we give the profile decomposition of \eqref{001} by
the strategy in
 \cite{AKIN-SM-2013, Nakanishi-2010,Kerrani2001,  MXZ-1,MXZ-2}. First, we give some notations for later use.

For $j,n\in \mathbb{N}$, we denote  $\Tnj$ by  $ \Tnj \vp(x)=\frac{1}{h^j_n}\vp(\frac{x-x^j_n}{h^j_n}) $, where $\vp\in H^1(\R^4)$
and $(x_n^j,h_n^j)\in\R^4\times (0,1]$.   We define $\tau^j_n=\frac{t^j_n}{(h^j_n)^2}$ and  the  multiplier
\beq
\snj=\frac{\Jap{(\hnj)^{-1} \nabla}^{-1} \Japnb  } {\hnj}.
\eeq

\subsection{Linear profile decomposition}
Now we give the linear profile decomposition for the solutions
sequence  to free Schr\"odinger equation in the inhomogeneous space
$H^1(\R^4)$.

\begin{prop}[Linear profile decomposition]\label{prop:profile}
Let $\vn(t,x)=e^{it\Delta}\vpn(x)$ be a sequence of the free
Schr\"odinger solutions with bounded $H^1(\R^4)$ norm. Then up to a
subsequence, there exist $K\in \set{1,2,\cdots,\infty}$,
$\set{\vp^j}_{j\in[1,K)}\subset H^1(\R^4)$ and
$\set{t_n^j,x_n^j,h_n^j}_{n\in \mathbb{N}}\subset \R\times\R^4\times
(0,1]$ satisfying

\beq\label{lpd}
\vn(t,x)=\sum_{j=1}^{k}\vn^j(t,x)+  e^{it\Delta} \omega^k_n,\; \text{for any }\; k\in [1,K),
\eeq
where
\beq
\begin{split}\label{lp-1}
     \vnj(t,x) &=    \Japnb^{-1}|\nabla|  e^{it \Delta}  \Tnj  e^{-i\taunj \Delta}   \Japnb|\nabla|^{-1}   \vp^j                 \\
               &=e^{it \Delta} \Tnj \snj e^{-i\taunj\Delta}\vp^j\\
               &=e^{i(t -t^j_n)\Delta} \Tnj \snj \vp^j
    \end{split}
\eeq
and $\set{\omega^k_n : k\in [1,K),\; n\geq 1 }\subset H^1(\R^4) $.
The error terms $\omega^k_n$ of the linear profile decomposition \eqref{lpd} have the properties:
\beq\label{lp-er}
\lim\limits_{k\ra K}\lim\limits_{n\ra \infty}\||\nabla|^{-1}\Japnb e^{it\Delta}\omega_n^k\|_{L_t^q(\R; L^r(\R^4))}=0, \text{ for any}\;
 \dot{H}^1 \text{-admissible pair}\;( q,r ),
\eeq and for any $j\leq k\in[1,K)$. \beq\label{lp-weak}
\lim\limits_{n\ra \infty} e^{i\taunj\Delta} (\Tnj)^{-1}
|\nabla|^{-1}\Japnb \omega_n^k = 0  \;\;\; \text{weakly in }\;\;
\dot{H}^1(\R^4). \eeq

For any $l<j\leq K,$ and $k\in [1,K)$, we have \beq\label{de-index}
             \lim\limits_{n\ra\infty} \left(\left|\log\frac{h^l_n}{h^j_n}\right|+\left|\frac{t_n^j-t_n^l}{(h^l_n)^2}\right|
             + \left|\frac{x_n^j-x_n^l}{h_n^l}\right|\right)=\infty,
 \eeq
\beq\label{de-mass}
 \lim\limits_{n\ra +\infty}\left| M(v_n(0))-\sum^{k}_{j=1} M(v^j_n(0))- M(\omega^k_n) \rg|=0, 
\eeq
\beq\label{de-energy}
\lim\limits_{n\ra +\infty}\left| E(v_n(0))-\sum^{k}_{j=1} E(v^j_n(0))- E(\omega^k_n) \rg|=0,
\eeq
and
\beq\label{de-K}
 \lim\limits_{n\ra +\infty}\left| K_{2,-1}(v_n(0))-\sum^{k}_{j=1} K_{2,-1}(v^j_n(0))- K_{2,-1}(\omega^k_n) \rg|=0.
\eeq Moreover, for each fixed $j$, the sequence $\set{h^j_n}_{n\in
N}$ is either going to $0$ or identically $1$ for all $n$.
\end{prop}
Note that, by $k\in [1,K)$ we mean that $1\leq k\leq K$ if $K<+\infty$ and $1\leq k<\infty$ if $K=\infty$.
\begin{proof}
  The proof is given by \cite[Lemma 2.2]{AKIN-SM-2013} based on \cite[Theroem 1.6]{Kerrani2001}, where the linear profile
  decomposition dealt with the sequences  in $\dot{H}^1$. And the property \eqref{lp-weak} is from \cite[lemma 2.10]{K-V-AJM2010}.

\end{proof}

We also have the following property for  a free Schr\"odinger solution  sequence if it is in $\mathcal{K}^+$.
\begin{cor}\label{cor:profile}
Suppose that $v_n(t,x)$ is a sequence of the free Schr\"odinger
solution with bounded $H^1(\R^4)$
 norm  satisfying that $$v_n(0)\in \mathcal{K}^+ \text{ and } E(v_n(0))\leq m_0< m.$$  
Let $$v_n(t,x)=\sum_{j=1}^{k} v_n^j(t,x)+ e^{it\Delta}\omega^k_n$$
 be the linear profile decomposition given by  Proposition 4.1. Then for sufficiently  large n and all $j,k\in[1,K)$, we have
$v^j_n(0),\; w^k_n(x)\in \mathcal{K}^+.$
  Moreover for all $j<K$, we have
  \beq \label{lpd-cor}
  0\leq \varliminf_{n\ra\infty} E(v^j_n(0))\leq \varlimsup_{n\ra\infty} E(v^j_n(0))\leq \varliminf_{n\ra\infty} E(v_n(0)) ,
  \eeq
  where the last inequality becomes equality  only if $K=1$ and $\omega^1_n(x)\ra 0$ in $\dot{H}^1$ as $n\ra \infty$.
\end{cor}
\begin{proof}
Since $v_n(0)\in \mathcal{K}^+$ and  $E(v_n(0))\leq m_0$, by Lemma \ref{lem:keylemma}, we have
$$\varlimsup_{n\ra\infty}\norm{\nabla v_n(0) }_{L^2(\R^4)} < \norm{\nabla W}_{L^2(\R^4)}.$$
And from the linear profile decomposition, up to a subsequence, we
have \beq\label{vo-bounded}
\lim_{n\ra\infty}\left(\|\vnj(0)\|_{H^1(\R^4)}^2
+\|\omega^k_n\|_{H^1(\R^4)}^2\right) \leq \lim_{n\ra\infty}
\|\vn(0)\|_{H^1(\R^4)}^2< \infty, \eeq
 and
$\|\vnj(0)\|_{\dot{H}^1(\R^4)}^2 +\|\omega^k_n\|_{\dot{H}^1(\R^4)}^2 <\| W\|_{\dot{H}^1(\R^4)}^2 $
for large $n$ and any $j<k\in [1,K)$. 

Then also by Lemma \ref{lem:keylemma}, we have $\vnj(0)$ and $\omega^k_n$
 positive energy, which together with \eqref{de-energy} implies \eqref{lpd-cor}.
 Thus, $v^j_n(0),\; w^k_n\in \mathcal{K}^+$ follows from the definition of $\K^+$.
\end{proof}

\subsection{Nonlinear Profile decomposition}

Now we are ready  to give the construction of  the nonlinear profile
decomposition. The strategy here is the same as the $3$D case in
\cite{MXZ-1} and higher dimensional case in \cite{AKIN-SM-2013,MXZ-2}.

Let $v_n(t,x)$ and $u_n(t,x)$ be  solutions to the  free
Schr\"odinger equation and the \eqref{001} respectively. And assume
that $v_n(t,x)$ and $u_n(t,x)$ have the same initial datum
$\vp_n(x)$.  By Proposition \ref{prop:profile}, we have the linear
profile decomposition \beq\label{lpd1}
\begin{split}
  v_n(t,x) = & \sum_{j=1}^{k} v^j_n(t,x)+ e^{it\Delta} \omega^k_n\\
       = &\sum_{j=1}^{k} e^{i(t -t^j_n)\Delta} \Tnj \snj \vp^j + e^{it\Delta} \omega^k_n .
\end{split}
\eeq
Let   $\unj(t,x)$ be the solution to \eqref{001} with the initial data $\unj(0,x)=\vnj(0,x)$.
  If we define $\Unj(t,x)$ by $\unj(t,x)=\Tnj  \snj \Unj\left(\frac{t-\tnj}{(\hnj)^2} \right)$, then $\Unj$ satisfies
\beq\label{}
\left\{
    \begin{array}{ll}
       (i\pa_t+\Delta) U^j_n  & = -(\snj)^{-1} f_1( \snj \Unj ) +    (\hnj)^{\frac{2}{3}} (\snj)^{-1} f_2(\snj \Unj )  \\
         \;\;\;\;\;U^j_n(-\taunj)  & = e^{-i\taunj\Delta}\vp^j
     \end{array}
\right. \eeq with $f_1(u)=|u|^2u$ and $f_2(u)=|u|^\frac{10}3u.$

Next, let $\{\Uij\}_{j\geq 1}$ be the solutions of the limit equation of above equation. More precisely, for each $j$,
we denote the limit of $\taunj$ and $\snj$ by
$$\tij=\lim_{n\ra \infty}\taunj\in [-\infty,+\infty] ,$$  and
$$\sij=\left\{
        \begin{array}{ll}
          1, & \hbox{\;\;\;$\hnj\equiv1$;} \\
          \frac{\Japnb}{|\nabla|}, & \hbox{$\hnj\ra 0$ as $n\ra\infty$.}
        \end{array}
      \right.
 $$
Then $\Uij$ satisfies

\beq\label{U;h=1}
\left\{
    \begin{array}{ll}
       (i\pa_t+\Delta) \Uij & = - f_1(\Uij ) +  f_2( \Uij )  \\
          \;\;\;\;\;\Uij(-\tij)  & = e^{-i\tij\Delta}\vp^j,
     \end{array}
\right.
\eeq
 when $\hnj\equiv1$, or
\beq\label{U;h=0}
\left\{
    \begin{array}{ll}
       (i\pa_t+\Delta) \Uij  & = -(\sij)^{-1} f_1( \sij \Uij )  \\
         \;\;\;\;\;\;\Uij(-\tij)  & = e^{-i\tij\Delta}\vp^j,
     \end{array}
\right.
\eeq
 when $\hnj\ra 0$ as $n\ra\infty$.
When  $\tij=\pm\infty$, we consider  \eqref{U;h=1} and
\eqref{U;h=0} as the final value problems, which means that
$\Uij(t,x)$ satisfies \beq\label{Uij=vp} \lim\limits_{n\ra\infty}
\norm{\Uij(-\taunj)-e^{-i\taunj\Delta} \vp^j}_{H^1(\R^4)}=0. \eeq

Based on this, we define the  nonlinear profile by
$\ujn(t,x)= \Tnj  \snj \Uij\left(\frac{t-\tnj}{(\hnj)^2} \right)$, which satisfies

\beq\label{u;h=1}
\left\{
    \begin{array}{ll}
      \;\;\; (i\pa_t+\Delta) \ujn & = - f_1(\ujn ) +f_2(\ujn )  \\
        \ujn\left(\tnj-\tau^j_\infty\right)  & =  e^{-i\tij\Delta}\vp^j,
     \end{array}
\right.
\eeq
 when $\hnj\equiv1$, or
\beq\label{u;h=0}
\left\{
    \begin{array}{ll}
        \;\;\; \;\;\;  (i\pa_t+\Delta) \ujn  & = -\frac{|\nabla|}{\Japnb} f_1\left(\frac{\Japnb}{|\nabla|}\ujn \right)  \\
         \ujn\lf(\tnj-\tau^j_\infty(\hnj)^2\rg)  & = e^{-i\tij\Delta}\vp^j,
     \end{array}
\right.
\eeq
 when $\hnj\ra 0$ as $n\ra\infty$. Then we define the nonlinear profile decomposition of $u_n(t,x)$ by
\beq\label{nlp}
\tilde{u}^{ k}_n:=\sum\limits_{j=1}^k \ujn+e^{it \Delta} \omega^k_n.
\eeq

Next we will show that each nonlinear profile such that  $\ujn\in \K^+,$
 if $u_n(0)\in \K^+$ for each $n$.

\begin{lem}\label{lem:nonprofi}
  Suppose $u_n(t,x)$ is a solution sequence of the \eqref{001} with the
  initial datum $u_n(0)$, which is  bounded in $H^1(\R^4)$ with $u_n(0)\in \K^+$.
  Then we have the nonlinear decomposition given by $$\tilde{u}^{ k}_n=\sum\limits_{j=1}^k \ujn+e^{it \Delta}
   \omega^k_n.$$
   Let $I^j$ be the maximal lifespan of $\Uij$, Then for $t\in I^j$,  we have that
  \beq \label{u-in-K}
   \sij \Uij(t)\in\left\{
                 \begin{array}{ll}
                   \K^+, & \hbox{$\hnj\equiv1$;} \\
                   \K^+_c, & \hbox{$\hnj\ra0 \text{ as } n\ra\infty$,}
                 \end{array}
               \right.
  \eeq
  where $\K^+_c$ is defined by
  $$\K^+_c:=\big\{\varphi\in\dot{H}^1(\R^4):~E_c(\vp)<E_c(W),~\|\nabla\vp\|_{L_x^2}^2\geq\|\vp\|_{L_x^4}^4\big\}.$$
\end{lem}
\begin{proof}

By Corollary \ref{cor:profile}, we have that
$\vnj(0)=e^{-i\tnj\Delta} \Tnj\snj\vp^j\in \K^+$. Then from the
definition of $\Tnj$ and the fact that $|\hnj|\leq 1,$ we have by
Lemma \ref{lem:keylemma}
$$\|e^{-i\taunj\Delta}\snj\vp^j\|_{\dot{H}^1(\R^4)}= \|e^{-i\tnj\Delta}\Tnj\snj\vp^j\|_{\dot{H}^1(\R^4)}< \|W\|_{\dot{H}^1(\R^4)}.$$

If $\hnj\equiv1$, then $\snj=\sij\equiv1$. Hence from
\eqref{Uij=vp}, we have  $\sij\Uij(-\taunj)\in \K^+$, for large $n$.
By Lemma \ref{lem:keylemma}, we have $\sij \Uij(t)\in \K^+$ for
$t\in I^j$.

On the other hand, if $\hnj\ra0$, then we have \beq
\|\sij\Uij(-\tij) \|_{\dot{H}^1(\R^4)} =\lim\limits_{n\ra\infty }
\norm{\sij e^{-i\taunj\Delta} \vp^j}_{\dot{H}^1(\R^4)}<
\|W\|_{\dot{H}^1(\R^4)}.
\eeq This implies $E_c(\sij\Uij(-\tij))>0$
by  Remark \ref{rem:positi}. And by the Sobolev embedding
$\dot{H}^1(\R^4)\hookrightarrow L^4(\R^d)$, we have
 $$E_c(\sij\Uij(-\tij))=\lim\limits_{n\ra\infty} E_c(\snj e^{-i\taunj\Delta} \vp^j )=\lim\limits_{n\ra\infty} E_c(v^j_n(0) )<E_c(W).$$
Hence, if $t\in I^j$, we have  $\sij \Uij(t)\in \K^+_c$, which  ends the proof.

\end{proof}

\begin{lem}\label{lem:nonproff12}
 Let $\Uij$ be constructed by above discussion and $I^j$ be the maximal lifespan of $\Uij$. Then there exists $j_0$ large enough such that the following property:
\beq
\sum\limits_{j\geq j_0} \norm{\Japnb \Uij}_{S(\R)}^2\leq \sum\limits_{j\geq j_0} \norm{\vp^j}^2_{H^1(\R^4)}<\infty.
\eeq

\end{lem}

\begin{proof} For any $\delta>0$, by the linear decomposition, there exist two constants $j_0$ and $n_0$ large enough with the property that:
If $n\geq n_0$, then we have
$$\sum\limits_{j\geq j_0}\norm{e^{-i\taunj\Delta}\vp^j }^2_{H^1(\R^4)}=\sum\limits_{j\geq j_0}\norm{\vp^j }^2_{H^1(\R^4)}=  \sum\limits_{j\geq j_0}
\norm{ \Tnj e^{-i\taunj\Delta}\snj \vp^j}^2_{H^1(\R^4)}<\delta.$$

If $\hnj\equiv1,$ by small data theory of \eqref{001}, we have $$\norm{\Japnb \Uij }^2_{S(\R)}\les \norm{\vp^j}^2_{H^1(\R^4)}.$$
If $\hnj\ra0$,   by the Strichartz estimates and the Mikhlin multiplier theorem, we have
\begin{align}\nonumber
\norm{\Japnb  \Uij}_{S(\R)} & \les \norm{\vp^j}_{H^1(\R^4)}+\norm{
|\nabla| f_1(\sij\Uij) }_{L^2_t(\R; L^{4/3}_x(\R^4))} \\\nonumber
                             & \les     \norm{\vp^j}_{H^1(\R^4)}+\norm{\sij\Uij}^2_{W_1(R)} \norm{\Japnb \Uij}_{L^6_t(\R; L^{12/5}_x(\R^4))}\\\nonumber
                             & \les  \norm{\vp^j}_{H^1(\R^4)}+ \norm{\Japnb  \Uij}_{S(\R)}^2\norm{\Japnb \Uij}^{1/3}_{L^2_t(\R; L^{4}_x(\R^4))}
                              \norm{\Japnb \Uij}^{2/3}_{L^\infty_t(\R; L^{2}_x(\R^4))}\\\label{bdd-S-U}
                             & \les  \norm{\vp^j}_{H^1(\R^4)}+ \norm{\Japnb  \Uij}_{S(\R)}^3 .
\end{align}
Hence,  by the continuity argument, for sufficiently small $ \delta>0$, we have $\norm{\Japnb  \Uij}_{S(\R)}\les \norm{\vp^j}_{H^1(\R^4)} $.
Combining this two cases together, we have $$\sum_{j\geq j_0}\norm{\Japnb  \Uij}^2_{S(\R)}\les \sum_{j\geq j_0}\norm{\vp^j}^2_{H^1(\R^4)} .$$
This ends the proof of this lemma.

\end{proof}

Recall $ST(I)=W_1(I)\cap W_2(I)$ in Section 2.1. Now we define the scattering size $\norm{\sij\Uij}_{ST^j(I)} $ of $\sij\Uij$  for each $j$ by
\begin{equation*}
ST^j(I) =\left\{
           \begin{array}{ll}
             W_1(I)\cap W_2(I), & \;\;\;\;\;\hbox{for $\hnj\equiv1$;} \\
             W_1(I), & \;\;\;\;\; \hbox{    $\hij=0$.}
           \end{array}
         \right.
\end{equation*}

\begin{lem}\label{lem:nonprosta}
  Let $k_0\in\N$ such that for any $1\leq j\leq k_0$,
\beq\label{nlp-01} \|\sij\Uij\|_{ST^j(I^j)} <\infty. \eeq Then we
have $I^j=\infty$ and \beq\label{nlp-02} \norm{\nabla
\ujn}_{V_0(\R)} + \norm{\ujn}_{ST(\R)} \les \|\Japnb \Uij\|_{S(\R)}
\les 1,\text{ for } 1\leq j\leq k_0. \eeq And there exists $B>0$
such that: for any given $k\in \N$ and $1\leq k\leq k_0$, there
exists $\ N_k\in\N$ s.t
 \beq \label{nlp-03} \sup_{n\geq N_k}
\left(\norm{\sum\limits_{j=1}^k \nabla
\ujn}_{V_0(\R)}+\norm{\sum\limits_{j=1}^k \ujn}_{ST(\R)} \right)
\leq B. \eeq Furthermore, if \eqref{nlp-01} holds for any $j
\in\N$, then, for any $p\in \{12/7,\;2\}$ and $j\geq 1$, we have
\beq\label{nlp-04} \lim\limits_{k\ra\infty} \lim\limits_{n\ra\infty}
\norm{\ujn |\nabla|^{s} e^{it\Delta }\omega^k_n
}_{L^p_{t,x}(\R^\times \R^4)} =0, \text{ and }
 s\in\{0,\;1\}.
\eeq
\end{lem}
\begin{proof}
First, $I^j=\R$ follows from the local well-posedness theory of
\eqref{U;h=1} and \eqref{U;h=0} and the scattering size condition of
\eqref{nlp-01}.

Next, for the first inequality in \eqref{nlp-02}, by employing the Mikhlin multiplier theorem, we have
\begin{align}\label{bdd-w1}
\norm{\ujn}_{W_1(\R)}=&\norm{\snj\Uij}_{W_1(\R)} \les \norm{\sij\Uij}_{W_1(\R)}\les \norm{\Japnb
\Uij}_{S(\R)},\\
\label{bdd-w2}
\norm{\ujn}_{W_2(\R)}  &\leq \norm{\ujn}^{1/2}_{W_1(\R)}
     \norm{\ujn}^{1/2}_{V_0(\R)}\\
    \nonumber &  \les  \norm{\Japnb \Uij}^{1/2}_{S(\R)}   (\hnj)^{1/2}\norm{\snj\Uij}^{1/2}_{V_0(\R)} \\
    \nonumber & \les \norm{\Japnb \Uij}_{S(\R)},\\
    \nonumber \text{ and } \\
\label{bdd-v0}\norm{\nabla \ujn}_{V_0(\R)}=&\norm{\nabla \snj\Uij}_{V_0(\R)}\les
\norm{\nabla\sij\Uij}_{S(\R)}\les \norm{\Japnb \Uij}_{S(\R)}.
\end{align}
Hence to finish \eqref{nlp-02}, it suffices to show
\begin{equation}\label{bdd-JsU}
\norm{\Japnb \Uij}_{S(\R)}\les 1.
\end{equation}

If $\hnj\ra 0$, by using the Strichartz estimates, \eqref{Uij=vp}, and \eqref{nlp-01}, we have
\begin{align}\label{bdd-JsU-h=0}
\norm{\Japnb  \Uij}_{S(\R)} & \les
\norm{\vp^j}_{H^1(\R^4)}+\norm{\Japnb (\sij)^{-1} f_1(\sij\Uij)
}_{L^2_t(\R; L^{4/3}_x(\R^4))} \\\nonumber
             & \les      \norm{\vp^j}_{H^1(\R^4)}+\norm{\sij\Uij}^2_{W_1(R)} \norm{\Japnb \Uij}_{L^6_t(\R; L^{12/5}_x(\R^4))}  \\\nonumber
             & \les  \norm{\vp^j}_{H^1(\R^4)}+ \norm{\Japnb \Uij}^{1/3}_{L^2_t(\R; L^{4}_x(\R^4))}
             \norm{\Japnb \Uij}^{2/3}_{L^\infty_t(\R; L^{2}_x(\R^4))}.
\end{align}
In this case,  similar to Lemma \ref{lem:nonproff12},
$\norm{\vp^j}_{H^1(\R^4)} \les 1$ and by Lemma \ref{lem:nonprofi},
we have $\norm{\Japnb \Uij}_{L^\infty_t(\R; L^{2}_x(\R^4))}\les 1$.
Thus, by the weighted H\"odelr inequality, we have
$$\norm{\Japnb \sij \Uij}_{S(\R)} \les 1.$$

On the other hand, if $\hnj\equiv 1$, then $\sij=1$. Similar to the former case, by the condition \eqref{nlp-01}, we have
\begin{equation}\label{bdd-JsU-h=1}
\begin{split}
\norm{\Japnb   \Uij}_{S(\R)} & \les 1+\norm{\Japnb \left[f_1( \Uij) +f_2( \Uij)\right] }_{L^2_t(\R; L^{4/3}_x(\R^4))} \\
            & \les  1 .
\end{split}
\end{equation}
By using  the standard continuity argument again, we have \eqref{bdd-JsU} in this case, which implies \eqref{nlp-02}.

Second, to prove \eqref{nlp-03},  we use 
$$ \abs{\Big|\sum_{1\leq j\leq k} \ujn\Big|^q -\sum_{1\leq j\leq k} \abs{\ujn}^q}\les_{k,q} \sum_{1\leq j\leq k}
\sum_{1\leq j'\leq k   \atop j'\neq j }\abs{\ujn}^{q-1}\abs{\ujjn}
,\;\; 1<q<\infty.   $$
to obtain
\begin{equation}\label{elem}
\Big\|\sum_{1\leq j\leq k}\ujn\Big\|^6_{W_1(\R)}\leq \sum_{1\leq j\leq k} \norm{\ujn}^6_{W_1(\R)}+ C_k \sum_{1\leq j\leq k}
\sum_{1\leq j'\leq k\atop j'\neq j }\int_{\R}\int_{\R^4} \abs{\ujn}^{5}\abs{\ujjn} dxdt.
\end{equation}
Without loss of generality, we assume that $k\geq j_0$ in Lemma \ref{lem:nonproff12}. Hence, we have
\begin{equation}\label{bdd-sum-u}\begin{split}
\sum_{1\leq j\leq k} \norm{\ujn}^6_{W_1(\R)}  \les  &  \sum_{1\leq j\leq j_0} \norm{\ujn}^6_{W_1(\R)} +\sum_{j_0<j\leq k} \norm{\ujn}^6_{W_1(\R)} \\
                            \les & \sum_{1\leq j\leq j_0} \norm{\Japnb \Uij}^6_{S(\R)} +\sum_{j_0<j\leq k}\norm{\vp^j}^2_{H^1}<+ \infty.
                        \end{split}
\end{equation}
On the other hand, by the almost  decoupling condition \eqref{de-index}, there exists $N_{k}\in \N$ such that
\begin{equation}\label{de-Uij}
\int_{\R}\int_{\R^4} \abs{\ujn}^{5}\abs{\ujjn} dxdt\leq \frac{1}{C_k  k^2}, \text{ for } j\neq j' \text{ and }n\geq N_{k}.
\end{equation}
Hence, from \eqref{elem}, \eqref{bdd-sum-u} and \eqref{de-Uij},
there exists $B_0>0$, which suffices that  for any $1\leq k\leq
k_0$, $\exists N'_k\in N$ such that
\begin{equation}\label{}
\sup_{n\geq N'_k} \Big\|\sum_{j=1}^k \ujn\Big\|_{W_1(\R)} \leq  B_0.
\end{equation}

Similarly, by the above  elementary inequality, Lemma
\ref{lem:nonproff12}  and the almost decoupling condition
\eqref{de-index}, we can prove that
 there exists $B_1>0$, which guarantees  that  for any $1\leq k\leq k_0$, $\exists N''_k\in N$ such that
\begin{equation}\label{}
\sup_{n\geq N''_k} \norm{\sum_{j=1}^k \nabla \ujn}_{V_0(\R)} +\sup_{n\geq N_k} \norm{\sum_{j=1}^k \ujn}_{W_2(\R)} \leq  B_1.
\end{equation}
Hence, by taking $B=\max\{B_0, B_1\}$ and $N_k=\max\{N'_k,N''_k \}$, we have \eqref{nlp-03}.

Finally, we turn to the proof of \eqref{nlp-04}.
 For the case $s=0$, using the H\"older inequality and the  Mikhlin multiplier theorem, we have as $n\ra \infty$
 \begin{equation}\label{nlp-s=0}
 \begin{split}
         &\norm{\ujn  e^{it\Delta}\onk }_{L^{12/7}_{t,x} \cap L^{2}_{t,x}} \\
    \les& \norm{\ujn}_{V_0(\R)} \left[\norm{ e^{it\Delta}\onk }_{W_1(\R)}+\norm{ e^{it\Delta}\onk }_{W_2(\R)} \right]    \\
     \les &\norm{\hnj\snj\Uij}_{L^3_{t,x}} \left[\norm{ e^{it\Delta}\onk }_{W_1(\R)}+\norm{  e^{it\Delta}\onk }^{1/2}_{W_1(\R)}
       \norm{e^{it\Delta}\onk }^{1/2}_{V_0(\R)} \right]\\
      \les & \norm{\Japnb\Uij}_{S(\R)}\norm{ e^{it\Delta}\onk }^{1/2}_{W_1(\R)} \left[ \norm{ e^{it\Delta}\onk }^{1/2}_{W_1(\R)}
      +\norm{\onk }^{1/2}_{H^1(\R^4)}\right] \ra 0,
 \end{split}
\end{equation}
where we used the fact that \eqref{lp-er} and $\onk\in H^1(\R^4)$ in
Proposition \ref{prop:profile}.

For the case $s=1$, from the fact that
$$\norm{\snj\Uij- \sij \Uij}_{W_1(\R)} \les
\left\{
                  \begin{array}{ll}
\norm{(\frac{\hnj}{\hnj+|\nabla|}) \Japnb\Uij}_{S(\R)}\ra 0, &\text{  as  } n\ra \infty, \text{ if } \hnj \ra 0 \\
    0, & \hnj\equiv 1,
\end{array}
 \right. $$
we have that  for any $\ep>0$, we can find $v^j\in
C_c^\infty(\R\times\R^4) $ and $N_j\in\N$ such that
$$ \norm{\snj\Uij-v^j}_{W_1(\R) }\leq \ep\quad \text{ for }\quad n>N_j . $$ Hence, by the H\"older inequality and the Strichartz estimates we have
\begin{equation}\label{nlp-s=1,1}
 \begin{split}
         &\norm{\ujn |\nabla| e^{it\Delta}\onk }_{ L^{2}_{t,x}(\R\times\R^4)} \\ 
    \les&  \norm{\snj\Uij \left[|\nabla| e^{i\;\centerdot\;\Delta}\onk\right ]
    \left(t(\hnj)^2+\tnj,x\hnj+\xnj\right)}_{ L^{2}_{t,x}(\R\times\R^4)}(\hnj)^2\\
    \les  &\;\norm{\snj\Uij-\vj}_{W_1(\R)}\norm{\onk }_{H^1(\R^4)}\\
           & +\norm{v^j \left[|\nabla| e^{i\;\centerdot\;\Delta}\onk\right ]\left(t(\hnj)^2+\tnj,x\hnj+\xnj\right)}_{
           L^{2}_{t,x}(\R\times\R^4)}(\hnj)^2.
 \end{split}
\end{equation}

Assume ${\rm supp} \;\vj\subset \{(t,x)\in\R\times\R^4:|t|\leq
T_j,|x|\leq R_j  \}$ and let
$\tilde{\omega}^{j,k}_n(x)=e^{i\tau^j_n\Delta}\onk(x\hnj+\xnj)$,
then by using Lemma 2.5 in \cite{K-V-AJM2010} and \eqref{lp-er} in
Proposition \ref{prop:profile},  we have
\begin{equation}\label{nlp-compact}
\begin{split}
   &   \norm{  \left[ |\nabla|e^{i(t(\hnj)^2+\tnj)\Delta}\onk\right](x\hnj+\xnj) }_{L^2_{t,x}(|t|\leq T_j,|x|\leq R_j)}(\hnj)^2 \\
  =& \hnj\norm{|\nabla|e^{it\Delta}\tilde{\omega}^{j,k}_n}_{L^2_{t,x}(|t|\leq T_j,|x|\leq R_j)} \\
  \les & \hnj(T_j)^{1/9}(R_j)^{1/18} \norm{e^{it\Delta}\tilde{\omega}^{j,k}_n}_{W_1(\R)}^{1/3} \norm{\nabla\tilde{\omega}^{j,k}_n}_{L^2_x(\R^4)}^{2/3} \\
 \les & (T_j)^{1/9}(R_j)^{1/18}  \norm{e^{it\Delta} \onk }_{W_1(\R)}^{1/3} \norm{\nabla\onk}_{L^2_x(\R^4)}^{2/3}  \ra 0,
 \end{split}
\end{equation}
as $k, n\ra\infty$.

On the other hand, for $p=12/7$, we define
 $$V^j=\left\{
             \begin{array}{ll}
                 0,    & \hbox{$\hnj\ra 0$;} \\
               \Uij, & \hbox{$\hnj\equiv1$.} 
             \end{array}
           \right.$$
Then we have  $$\norm{(\hnj)^{1/2}\snj\Uij-V^j}_{W_2(\R)}\les \norm{ |\nabla|^{1/2}(\hnj)^{1/2}\snj\Uij-|\nabla|^{1/2}V^j}_{V_2}\ra 0 , $$
 as $n\ra \infty$.
From this we can find $v^j\in C^\infty_c(\R\times\R^4)$ with the property that for any $\eta>0 $, there exists $N_j\in \N$, such that
 $\norm{(\hnj)^{1/2}\snj\Uij-v^j}_{W_2(\R)} <\eta $, if $n>N_j$.

By the H\"older inequality and the Strichartz estimates, we have
 \begin{equation}\label{nlp-compact2}
 \begin{split}
         &\norm{\ujn |\nabla| e^{it\Delta}\onk }_{L^{12/7}_{t,x} } \\
    \les& \norm{ (\hnj)^{1/2}\snj\Uij \cdot |\nabla|e^{i(t(\hnj)^2+\hnj)\Delta}\onk(x\hnj+\xnj)  }_{ L^{12/7}_{t,x}} (\hnj)^2  \\
     \les & \eta\norm{|\nabla|e^{i(t(\hnj)^2+\hnj)\Delta}\onk(x\hnj+\xnj)}_{L^3_t,x}(\hnj)^2 \\
      &  + \norm{v^j\left[|\nabla|e^{it(\hnj)^2+\tnj\Delta}\onk\right](x\hnj+\xnj)}_{L^{12/7}_{t,x}(\R\times\R^4)}(\hnj)^2 \\
      \les & \eta \norm{|\nabla| \onk  }_{ L^{2}(\R^4)} +  \norm{v^j\left[|\nabla|e^{it(\hnj)^2+\tnj\Delta}\onk\right]
      (x\hnj+\xnj)}_{L^2_{t,x}(\R\times\R^4)}(\hnj)^2.
 \end{split}
\end{equation}
Similar to \eqref{nlp-compact}, we have the last term of
\eqref{nlp-compact2} tends to $0$, as $k,n\ra\infty$, which together
with \eqref{nlp-s=1,1}, \eqref{nlp-compact}  implies \eqref{nlp-04}.
This ends the proof of Lemma \ref{lem:nonprosta}.

\end{proof}

\begin{prop}\label{prop:scatnon}
Let $\unj$ and $\Uij $ be as above. Assume that \beq \label{un-1}
\|\sij\Uij\|_{ST^j(I^j)}<\infty, \text{\;\; holds for any\;\; }
j\geq 1. \eeq Then we have \beq\label{un<infty} \norm{u_n}_{ST(\R)}
< \infty,\;\; \text{ for sufficient large\;} n. \eeq

\end{prop}
\begin{proof}
To prove \eqref{un<infty}, we will use the perturbation theory
(Proposition \ref{prop:stability}). More precisely, we need to prove that
the nonlinear profile $\tilde{u}^k_n$ constructed above is actually
a sequence of approximation solutions.

First, by Lemma \ref{lem:nonprosta} and  our assumption of $\Uij$ in
\eqref{un-1}, we have that $\ujn$ is a global solution and $\utkn$
is a function of
 $t$ defined on the whole $\R$.

Next, from the inequalities \eqref{nlp-03} in Lemma
\ref{lem:nonprosta} and $H^1(\R^4)$ bounded condition of
$\omega^k_n(x)$ in Corollary \ref{cor:profile}, we have that there exists a constant $B>0$
satisfying the property: there exists $N_{1,k}>0$, such that \beq
\sup\limits_{n\geq N_{1,k}} \norm{\tilde{u}^k_n}_{ST(\R)} \leq B.
\eeq
In fact, the linear profile decomposition implies
\beq
\begin{split}
\norm{u_n(0)-\utkn(0) }_{H^1(\R^4)} &\leq \sum\limits_{j=1}^k   \norm{ v^j_n(0)-\ujn(0)}_{H^1(\R^4)}\\
                                    & \les  \sum\limits_{j=1}^k \norm{e^{-\taunj\Delta}\vp^j-\Uij(-\taunj) }_{H^1(\R^4)}.
\end{split}
\eeq From this and  \eqref{Uij=vp}, we have that for any $k\in\N $,
there exists $N_{2,k}>0$ such that \beq \sup\limits_{n\geq
N_{2,k}}\norm{u_n(0)-\utkn(0) }_{H^1(\R^4)} \leq 1. \eeq Similarly,
by the Strichartz estimates and \eqref{Uij=vp}, for any $k\in \N$,
we can find $N_{3,k}\in \N$ such that \beq \sup\limits_{n\geq
N_{3,k}}\norm{ \Japnb e^{it\Delta} \left( u_n(0)-\utkn(0)
\right)}_{V_2(\R)} \leq \delta, \eeq where the $\delta>0$ is from
the perturbation theory and
depends on $B$ and the uniform bound of $\|u_n(0)\|_{H^1((\R^4))}$.

Hence, by Proposition  \ref{prop:stability}, \eqref{un<infty} holds
if we could show  that there exist $k_0,\;N_0 \in \N$ such that if $k\geq k_0$ and $n\geq N_0$,
  \beq\label{p-t} \norm{\Japnb\left[
(i\pa_t+\Delta) \utkn + f_1(\utkn)-f_2(\utkn)
\right]}_{L_{t,x}^\frac{3}{2}(\R\times\R^4)} \leq \delta. \eeq
To do this, if we define
$f(u)= f_1(u)-f_2(u)$, then we have \beq
\begin{split}
(i\pa_t+\Delta) \utkn + f(\utkn)= &f(\utkn)-f\left(\sum_{j=1}^k \ujn\right)\\
                                  & +\sum\limits_{j=1}^k (i\pa_t+ \Delta)\ujn+f\left(\sum_{j=1}^k \ujn\right).
\end{split}
\eeq
Hence, we just need to show that
 \begin{equation}\label{p-t-1}
\lim\limits_{k\ra \infty} \lim\limits_{n\ra\infty}  \norm{ \Japnb \left[ f\left(\utkn\right)-f\left(\sum_{j=1}^k
\ujn\right)\right]}_{L_{t,x}^\frac{3}{2}(\R\times\R^4)}=0,
\end{equation}
and  \begin{equation}\label{p-t-2}
 \lim\limits_{k\ra \infty} \lim\limits_{n\ra\infty} \norm{ \Japnb \left[ \sum\limits_{j=1}^k (i\pa_t+ \Delta)\ujn+f\left(\sum_{j=1}^k \ujn\right)
   \right]      }_{L_{t,x}^\frac{3}{2}(\R\times\R^4)}=0.
\end{equation}

First, we will prove \eqref{p-t-1}.  
 Let $u(x)$ and $w(x)$ be two functions defined on $\R^4$, then if  $1<r\leq 2$,  we have
\begin{equation}\label{iq-1}
\big||u+w|^r(u+w)-|u|^ru\big| \les (|u|^r +|w|^r)|w|,
\end{equation}
and
\begin{equation}\label{iq-2}
\abs{\nabla\left(|u+w|^r(u+w)-|u|^ru \right)} \les |u|^r|\nabla w|+ (|\nabla u| +|\nabla w|)|w|^r.
\end{equation}
By these two inequalities and the H\"older inequality, to prove
\eqref{p-t-1}, we just need to estimate
\begin{equation}\label{p-t-3}
\begin{split}
   & \norm{ \left[  |e^{it\Delta} \omega^k_n|^{\frac{4}{3}} + |e^{it\Delta} \omega^k_n |^{2} \right]  |e^{it\Delta} \omega^k_n|
    }_{L_{t,x}^\frac{3}{2}(\R\times\R^4)} \\
+ &\norm{ \left[  |e^{it\Delta} \omega^k_n|^{\frac{4}{3}}+ |e^{it\Delta} \omega^k_n |^{2} \right]  |\nabla e^{it\Delta} \omega^k_n|
 }_{L_{t,x}^\frac{3}{2}(\R\times\R^4)},
\end{split}
\end{equation}
\begin{equation}\label{p-t-4}
\norm{ \left[  |e^{it\Delta} \omega^k_n|^{\frac{4}{3}}+ |e^{it\Delta} \omega^k_n |^{2} \right]  \Big| \sum_{j=1}^k \nabla\ujn\Big|
 }_{L_{t,x}^\frac{3}{2}(\R\times\R^4)},
\end{equation}
and
\begin{equation}\label{p-t-5}
\begin{split}
&\norm{\left[  |\sum_{j=1}^k \ujn |^{\frac{4}{3}}+ |\sum_{j=1}^k \ujn |^{2} \right] |e^{it\Delta} \omega^k_n| }_{L_{t,x}^\frac{3}{2}(\R\times\R^4)} \\
+ & \norm{\left[  |\sum_{j=1}^k \ujn |^{\frac{4}{3}}+ |\sum_{j=1}^k \ujn |^{2} \right] |\nabla e^{it\Delta} \omega^k_n|
 }_{L_{t,x}^\frac{3}{2}(\R\times\R^4)}.
\end{split}
\end{equation}

Now deal with \eqref{p-t-3}. Using the H\"older inequalities, we
have
\begin{equation}\label{p-t-6}
  \norm{|e^{it\Delta}\omega^k_n|^{4/3} |e^{it\Delta}\omega^k_n| }_{L_{t,x}^\frac{3}{2}(\R\times\R^4)}\leq \norm{e^{it\Delta}\omega^k_n}^{2/3}_{W_1(\R)}
  \norm{e^{it\Delta}\omega^k_n}^{5/3}_{V_0(\R)},
\end{equation}
\begin{equation}\label{p-t-7}
 \norm{ |e^{it\Delta}\omega^k_n|^{2} |e^{it\Delta}\omega^k_n| }_{L_{t,x}^\frac{3}{2}(\R\times\R^4)}   \leq  \norm{e^{it\Delta}\omega^k_n}^2_{W_1(\R)}
  \norm{ e^{it\Delta}\omega^k_n }_{V_0(\R)},
\end{equation}
\begin{equation}\label{p-t-8}
\begin{split}
  & \norm{|e^{it\Delta}\omega^k_n|^{4/3} |\nabla e^{it\Delta}\omega^k_n |}_{L_{t,x}^\frac{3}{2}(\R\times\R^4)}\\
  \leq  & \norm{e^{it\Delta}\omega^k_n}^{2/3}_{W_1(\R)} \norm{ e^{it\Delta}\omega^k_n}^{2/3}_{V_0(\R)}\norm{ \nabla e^{it\Delta}\omega^k_n}_{V_0(\R)} ,
\end{split}
\end{equation}
and
\begin{equation}\label{p-t-9}
 \norm{|e^{it\Delta}\omega^k_n|^{2} |\nabla e^{it\Delta}\omega^k_n| }_{L_{t,x}^\frac{3}{2}(\R\times\R^4)}\leq \norm{e^{it\Delta}
 \omega^k_n}^{2}_{W_1(\R)} \norm{\nabla e^{it\Delta}\omega^k_n}_{V_0(\R)}.
\end{equation}
Then, by the Strichartz estimates and  combining the above four
inequalities with the fact that  $\|\omega^k_n\|_{H^1(\R^4)}\leq C$,
we have \beq \label{p-t-10} \eqref{p-t-3} \les \left[
\norm{e^{it\Delta}\omega^k_n}^{2/3}_{W_1(\R)} +
\norm{e^{it\Delta}\omega^k_n}^2_{W_1(\R)} \right] .\eeq Hence, by the
property \eqref{lp-er} of $\omega^k_n$, we obtain that
$\eqref{p-t-3}$ tends to $0$, as $n,\;k\ra\infty$.

From the same method of \eqref{p-t-3} and the fact
$$\sup_{n\geq N_k}\norm{\sum_{j=1}^k \nabla \ujn}_{V_0(\R)}\les B,$$
in \eqref{nlp-03} from Lemma \ref{lem:nonprosta}.
we can deal with \eqref{p-t-4}.



To deal with \eqref{p-t-5}, we consider the terms having the form of
\begin{equation}\label{p-t-11}
 \norm{\Big|\sum_{j=1}^k \ujn \Big|^{4/3} \cdot \left||\nabla|^s e^{it\Delta} \omega^k_n \right| }_{L_{t,x}^\frac{3}{2}(\R\times\R^4)}+ \norm{\Big|\sum_{j=1}^k \ujn \Big|^2 \cdot \left||\nabla|^s e^{it\Delta} \omega^k_n \right| }_{L_{t,x}^\frac{3}{2}(\R\times\R^4)},
\end{equation}
where $ s\in\{ 0,1\}$.

By the H\"older inequality and \eqref{nlp-03} in Lemma \ref{lem:nonprosta}, we have
\begin{equation}\label{p-t-12}\begin{split}
                              &  \Big\||\sum_{j=1}^k \ujn |^{4/3} \cdot \left||\nabla|^s e^{it\Delta} \omega^k_n \right|
                               \Big\|_{L_{t,x}^\frac{3}{2}(\R\times\R^4)}  \\
                             \les   &   \Big\|\sum_{j=1}^k \ujn\Big\|^{1/3}_{W_2(\R)} \cdot \Big\| \sum_{j=1}^k \ujn   |\nabla|^s e^{it\Delta}
                             \omega^k_n\Big\|_{L^{12/7}_{t,x}(\R\times\R^4)} \\  
                             \les & \Big\|  \sum_{j=1}^k \ujn   |\nabla|^s e^{it\Delta} \omega^k_n\Big\|_{L^{12/7}_{t,x}(\R\times\R^4)} ,
                              \end{split}
\end{equation}
and
\begin{equation}\label{p-t-13}
\begin{split}
  &  \Big\||\sum_{j=1}^k \ujn |^2 \cdot \left||\nabla|^s e^{it\Delta} \omega^k_n \right| \Big\|_{L_{t,x}^\frac{3}{2}
(\R\times\R^4)}  \\
  \les   &   \Big\|\sum_{j=1}^k \ujn \Big\|_{W_1(\R)}  \cdot  \Big\| \sum_{j=1}^k \ujn   |\nabla|^s e^{it\Delta}\omega^k_n \Big\|_{L^{2}_{t,x}(\R\times\R^4)} \\ 
 \les & \Big\| \sum_{j=1}^k \ujn   |\nabla|^s e^{it\Delta} \omega^k_n \Big\|_{L^{2}_{t,x}(\R\times\R^4)}.
                              \end{split}
\end{equation}

Hence, to prove \eqref{p-t-1}, we just need show
\begin{equation}\label{p-t-14}
\lim\limits_{k\ra \infty} \lim\limits_{n\ra\infty}    \Big\|  \sum_{j=1}^k \ujn   |\nabla|^s e^{it\Delta} \omega^k_n \Big\|_{L^{p}_{t,x}(\R\times\R^4)}=0,
\end{equation}
 where $p\in\{12/7,\;2\}$ and $s\in\{0,\;1\}$.

By Lemma \ref{lem:nonproff12}, for any $\ep>0$, there exists $J(\ep)>0$ such that
\begin{equation}\label{p-t-15}
\sum\limits_{J(\ep)\leq j\leq k} \norm{\ujn}_{ST(\R)}^2 \leq  \sum\limits_{j\geq J(\ep)} \|\Japnb \Uij \|_{S(\R)}^2  <\frac{1}{4}\ep^2.
\end{equation}
Using the same method of proving \eqref{nlp-03}, by the almost decoupling condition \eqref{de-index}, we have can find $N_{k,\ep}>0$ such that:
 if $n>N_{k,\ep}$, then
\begin{equation}\label{p-t-16}
   \Big\|\sum\limits_{J(\ep)\leq j\leq k} \ujn  \Big\|_{ST(\R)} \leq  \left[\sum\limits_{J(\ep)\leq j\leq k} \norm{\ujn}_{ST(\R)}^2\right]^{1/2}
  +\frac{1}{2}\ep   <\ep,
\end{equation}
where in the last step, we used the boundedness of  $\norm{\vp^j}_{H^1}$.

Then, for  $p\in\{12/7,\;2\}$ and $s\in\{0,\;1\}$, by the H\"older
inequality and the Strichartz estimates, we have
\begin{equation}\label{p-t-17}
  \norm{  \sum_{J(\ep)\leq j\leq k} \ujn   |\nabla|^s e^{it\Delta} \omega^k_n}_{L^{p}_{t,x}(\R\times\R^4)} \leq \norm{ \sum_{J(\ep)\leq j\leq k}
   \ujn }_{ST(\R)} \norm{|\nabla|^s e^{it\Delta} \omega^k_n}_{V_0(\R)} \les \ep.
\end{equation}

On the other hand, for  $p\in\{12/7,\;2\}$ and $s\in\{0,\;1\}$, by
\eqref{nlp-04} in Lemma \ref{lem:nonprosta}, we have
\beq\label{p-t-18} \lim\limits_{k\ra \infty}
\lim\limits_{n\ra\infty}   \norm{ \sum_{1\leq j\leq J(\ep)} \ujn
|\nabla|^s e^{it\Delta} \omega^k_n}_{L^{p}_{t,x} (\R\times\R^4)}=0,
\eeq which together with \eqref{p-t-17} implies \eqref{p-t-14}. This
finishes the proof of \eqref{p-t-1}.\vspace{0.3cm}

\emph{Proof of \eqref{p-t-2}}. Recall that $\ujn$ satisfy that
\eqref{u;h=1} and \eqref{u;h=0}, then we have
\begin{align}
  \eqref{p-t-2} \les & \sum_{s=0,1} \norm{ |\nabla|^s \left[ \sum_{j=1}^k f\left(\ujn \right)-f\left(\sum_{j=1}^k \ujn \right) \right]
   }_{L^{3/2}_{t,x}(\R\times\R^4)}  \label{pt-2-1} \\
&+ \sum_{s=0,1} \Big\||\nabla|^s   \sum_{1\leq j \leq k\atop \hnj\ra 0} \left[  (\sij)^{-1}  f_1\left(\sij\ujn\right)-f_1\left( \ujn \right) \right]
  \Big\|_{L^{3/2}_{t,x}(\R\times\R^4)}  \label{pt-2-2}\\
&+  \sum_{s=0,1} \Big\| |\nabla|^s \sum_{1\leq j \leq k\atop \hnj\ra 0 } f_2\left(\ujn\right) \Big\|_{L^{3/2}_{t,x}(\R\times\R^4)}.\label{pt-2-3}
\end{align}

For \eqref{pt-2-1}, using a elementary inequality,
 \eqref{pt-2-1} can be estimated by
\beq\label{2-1-1}
\sum_{s=0,1} \sum_{1\leq j,j'\leq k \atop j\neq j' }\norm{   \abs{\ujn}^{2} |\nabla|^s \ujjn }_{L^{3/2}_{t,x}(\R\times\R^4)}
\eeq
plus
\beq\label{2-1-2}
\sum_{s=0,1} \sum_{1\leq j,j'\leq k \atop j\neq j' }\norm{   \abs{\ujn}^{4/3} |\nabla|^s \ujjn }_{L^{3/2}_{t,x}(\R\times\R^4)}.
\eeq

For \eqref{2-1-1}, we use density.
By \beq \norm{\snj\Uij-\sij\Uij}_{W_1(\R)}\les
\norm{\frac{\min(1-\hnj,\hnj)}{\hnj+|\nabla|}\Japnb\Uij}_{S(\R)}\ra
0 \eeq and \beq
\norm{\nabla\left[\snj\Uij-\sij\Uij\right]}_{V_0(\R)}\les
\norm{\frac{\min(1-\hnj,\hnj)}{\hnj+|\nabla|}\Japnb\Uij}_{S(\R)}\ra
0, \eeq we have for any $\ep>0$, there exists  $U^j,\; V^j\in
C^\infty_c(\R\times\R^4)$  such that  for sufficiently large $n$,
 \beq \norm{\left[\Tnj\snj\Uij-\Tnj
U^j\right]\left(\frac{t-\tnj}{(\hnj)^2}\right)}_{W_1(\R)} +
\norm{\nabla\left[\Tnj\snj\Uij-\Tnj
V^j\right]\left(\frac{t-\tnj}{(\hnj)^2}\right)}_{V_0(\R)}<\ep ,\eeq
 by density. Hence \beq \eqref{2-1-1} \les \ep +\sum_{s=0,1} \sum_{1\leq j,j'\leq
k \atop j\neq j' }\norm{|\Tnj
U^j|^{2}\left(\frac{t-\tnj}{(\hnj)^2}\right) \cdot |\nabla|^s
 \Tnj
 V^{j'}\left(\frac{t-\tnjj}{(\hnjj)^2}\right)}_{L^{3/2}_{t,x}(\R\times\R^4)}.
\eeq If $\frac{\hnj}{\hnjj}\ra \infty$, we have
\begin{align}
& \norm{|\Tnj U^j|^{2}\left(\frac{t-\tnj}{(\hnj)^2}\right) \cdot
|\nabla| \Tnjj V^{j'}\left(\frac{t-\tnjj}{(\hnjj)^2}
\right)}_{L^{3/2}_{t,x}(\R\times\R^4)}  \\\nonumber \les & (\hnjj)^2
(\hnj)^{-2}\norm{|U^j|^{2}
\left(\frac{t(\hnjj)^2+\tnjj-\tnj}{(\hnj)^2},
\frac{x\hnjj+\xnjj-\xnj}{\hnj}\right) \cdot \left[|\nabla|
V^{j'}\right] (t,x) }_{L^{3/2}_{t,x}(\R\times\R^4)}\\\nonumber \les
& (\hnjj)^2 (\hnj)^{-2}\norm{U^j}_{L^\infty_{t,x}(\R\times\R^4)}
\norm{ |\nabla| V^{j'} }_{L^{3/2}_{t,x}(\R\times\R^4)}^{2}\ra 0,
\end{align}
as $n\ra \infty$.

If $\frac{\hnj}{\hnjj}\ra 0$, we have
\begin{align}
& \norm{|\Tnj U^j|^{2}\left(\frac{t-\tnj}{(\hnj)^2}\right) \cdot
|\nabla| \Tnjj V^{j'}\left(\frac{t-\tnjj}{(\hnjj)^2}
\right)}_{L^{3/2}_{t,x}(\R\times\R^4)}  \\\nonumber \les & (\hnj)^2
(\hnjj)^{-2}\norm{|U^j|^{2}(t,x) \cdot \left[|\nabla| V^{j'}\right]
\left(\frac{t(\hnj)^2+\tnj-\tnjj}{(\hnjj)^2},
 \frac{x\hnj+\xnj-\xnjj}{\hnjj}\right)}_{L^{3/2}_{t,x}(\R\times\R^4)}\\\nonumber
\les & (\hnj)^2 (\hnjj)^{-2}\ra 0,
\end{align}
as $n\ra \infty$.

If there exists $C_j$ such that $\frac{\hnj}{\hnjj}+\frac{\hnjj}{\hnj}<C_j$ for any $n$,  then by the almost decoupling condition \eqref{de-index}, we have
 $$ \abs{\frac{\tnjj-\tnj}{(\hnj)^2}}+\abs{\frac{\xnjj-\xnj}{\hnj}}\ra\infty, \text{ as } n\ra \infty,$$
where by symmetry we assume that  $j<j'$.

Then for sufficiently large $n$, we have
\begin{align}
& \norm{|\Tnj U^j|^{2}\left(\frac{t-\tnj}{(\hnj)^2}\right) \cdot
|\nabla| \Tnjj V^{j'}\left(\frac{t-\tnjj}{(\hnjj)^2}\right)
}_{L^{3/2}_{t,x}(\R\times\R^4)}  \\\nonumber \les & (\hnjj)^2
(\hnj)^{-2}\norm{|U^j|^{2}
\left(\frac{t(\hnjj)^2+\tnjj-\tnj}{(\hnj)^2},
\frac{x\hnjj+\xnjj-\xnj}{\hnj}\right) \cdot \left[|\nabla|
V^{j'}\right] (t,x) }_{L^{3/2}_{t,x}(\R\times\R^4)}\\\nonumber =& 0,
\end{align}
where we used the fact that $U^j,\;V^j\in C^\infty_c(\R\times\R^4)$.

On the other hand, we need to consider
\beq
\sum_{j,j'}\norm{|\ujn|^2\ujjn}_{L^{3/2}_{t,x}(\R\times\R^4)}.
\eeq

If $\hnj\equiv 1$, then we have $\norm{\snj\Uij}_{S(\R)}=\norm{\Uij}_{S(\R)}\leq C $.
If $\hnj\ra 0$, then we have that $\norm{\ujn}_{V_0(\R)}=(\hnj)^2\norm{\snj\Uij}_{V_0}\les \hnj \norm{\Japnb \Uij}_{S(\R)}\ra 0$.
Hence by the similar argument as before, we have  for any $\ep>0$, there exists $U^j,V^j\in C_c(\R\times\R)$ such that
$$ \norm{\ujn-\Tnj U^j}_{W_1(\R)}+\norm{\ujn-\Tnj V^j}_{V_0(\R)}<\ep$$
 for sufficiently large $n$.

Hence for large $n$, we have
\begin{align}\nonumber
&\sum_{j,j'}\norm{|\ujn|^2\ujjn}_{L^{3/2}(\R\times\R^4)} \\\nonumber
       =   &  \sum_{j,j'} \norm{|\ujn|^2 [\ujjn-\Tnjj V^{j'}]}_{L^{3/2}(\R\times\R^4)}+\norm{\abs{\ujn}^2\Tnjj V^{j'}}_{L^{3/2}(\R\times\R^4)}
       \\\nonumber
\les &  \sum_{j,j'} \norm{\ujn}_{W_1(\R)}^2 \norm{\ujjn-\Tnjj
V^{j'}}_{V_0(\R)}+\norm{\abs{\ujn}^2\Tnjj
V^{j'}}_{L^{3/2}(\R\times\R^4)} \\\nonumber
 \les &\; \ep +\sum_{j,j'} (\hnjj)^3(\hnj)^{-2}\norm{\abs{\snj\Uij}^2\left(\frac{t(\hnjj)^2+\tnjj-\tnj}{(\hnj)^2},\frac{x\hnjj+\xnjj-\xnj}{\hnj}\right)
 V^{j'}(t,x)}_{L^{3/2}(\R\times\R^4)}\\\nonumber
 \les & (1+\hnjj)\ep +     (\hnjj)^3(\hnj)^{-2}\norm{\abs{U^j}^2\left(\frac{t(\hnjj)^2+\tnjj-\tnj}{(\hnj)^2},\frac{x\hnjj+\xnjj-\xnj}{\hnj}\right)
 V^{j'}(t,x)}_{L^{3/2}(\R\times\R^4)}              \\
 \les&\; \ep,
\end{align}
where in the last inequality we used the  almost decoupling condition \eqref{de-index} and the fact that $U^j\in C_c^\infty(\R\times\R^4)$.

Note that $$\norm{\ujn}_{W_2(\R)}+\norm{\Japnb\ujn}_{V_0(\R)}\les \norm{\Japnb\Uij}_{S(\R)}. $$
Hence,  we can use the similar argument to prove that $\eqref{2-1-2}$ tends to $0$ as $n\ra \infty$.

%

We consider the contribution of \eqref{pt-2-2}, by the analogue estimates in Lemma 8.9 of
\cite{Nawa2015}. 
Indeed, if $j$ such that $\hnj\ra 0$,  then we
have \beq
\begin{split}
  &\norm{\Japnb \left[ (\sij)^{-1} f_1\brc{\sij\ujn}-f_1\brc{\ujn}\right]}_{L^{3/2}_{t,x}(\R\times\R^4)}\\
           =& \norm{|\nabla| f_1\brc{\sij\Uij}-\brc{\hnj+|\nabla|} f_1(\snj\Uij)}_{L^{3/2}_{t,x}(\R\times\R^4)} \\
     \les & \norm{|\sij\Uij|^2\Japnb\Uij-|\snj\Uij|^2 \Japnb \Uij}_{L^{3/2}_{t,x}(\R\times\R^4)}\\
  \les &   \norm{\Japnb\Uij}_{S(\R)} \norm{\sij\Uij}_{W_1(\R)}\norm{\sij\Uij-\snj\Uij}_{W_1(\R)}\\
\les & \norm{\frac{\hnj}{\hnj+|\nabla|}\Japnb\Uij }_{W_1(\R)}\ra 0,
 \end{split}
\eeq
as $n\ra\infty,$ which implies $\eqref{pt-2-2}\ra 0$.

For \eqref{pt-2-3}, we have
\begin{equation}\label{}
 \begin{split}
   \eqref{pt-2-3} & \les \sum_{1\leq j\leq k \atop \hnj\ra0}(\hnj)^{2/3} \norm{ \Japnb f_2\left(\snj \Uij\right)}_{L^{3/2}_{t,x}(\R\times\R^4)} \\
     & \les  \sum_{1\leq j\leq k \atop \hnj\ra0}(\hnj)^{2/3} \norm{ \snj\Uij}_{W_2(\R)}^{4/3} \left[\norm{|\nabla|\snj\Uij }_{V_0(\R)}
     + \hnj\norm{\snj\Uij }_{V_0(\R)} \right]\\
     & \les  \sum_{1\leq j\leq k \atop \hnj\ra0}  \norm{  \frac{(\hnj)^{1/2}|\nabla|^{1/2}}{\hnj+|\nabla|}  \Japnb\Uij}^{4/3}_{V_2(\R)
     }\norm{\Japnb\Uij}_{V_0(\R)}\ra 0,
 \end{split}
\end{equation}
as $n\ra \infty $, which ends the proof of \eqref{p-t-2}.

\end{proof}

\section{GWP and scattering of $\K^+$}
In this part, we will prove Theorem \ref{thm:main} by the argument
of contradiction.

Firstly, we define the the minimal blowup energy of \eqref{001}.
For fixed $C\in(0,\infty)$, and  any $E\in (0,\infty)$, let
\beq
A_C(E)=\sup\Big\{\norm{u}_{ST(I)}:  u_0\in \K^+ \text { and } E(u_0) \leq E, M(u_0)\leq C \Big\},
\eeq
where $u$ is the solution to \eqref{001} on the maximal lifespan time interval $I$.
Suppose $E_C^*$ such that $A_C(E^*_C)=\infty$ is the minimal blowup energy.

Next, we will prove that: If $E_C^*<m$, then there exists a critical element with energy equal to $E_C^*$ and has some special properties.

Finally, by making use of the interaction Morawetz estimates similar in \cite{Dodson-2014},
 we will prove that the critical element equal to zero, which contradicts with the small data theory.
\subsection{Existence of the critical element}

\begin{lem}\label{lem:5.1} Let $\{u_n\}$ be a sequence of solutions of \eqref{001} in $\K^+$ on $I_n\subset \mathbb{R}$ satisfying
$M(u_n)\leq C$ and \beq\label{u-condition} E(u_n)\rightarrow
E_C^*<m,\;\|u_n\|_{ST(I_n)}\rightarrow\infty\;\; \text{ as }
n\ra\infty. \eeq Then there exists a global solution $u_c$ of
$(CNLS)$ in $\K^+$ satisfying \beq\label{uc-condition}
E(u_c)=E^*_C<m,\; \|u_c\|_{ST(\R)}=\infty. \eeq In addition, there
exist  a sequence $(t_n,x_n)\in(\mathbb{R}\times\mathbb{R}^4)$ and
$\vp\in H^1(\R^4)$ such that, up to a subsequence, we have as
$n\ra\infty$, \beq\label{u-compact}
 \norm{u_n(0,x)-e^{-it_n\Delta}\vp(x-x_n)}_{\dot{H}^1(\R^4)}\ra 0.
\eeq

\end{lem}

\begin{proof}
For the sequence $\{u_n\}$ satisfying \eqref{u-condition}, by the time translation symmetry of \eqref{001} we can assume that for each $n$,
$0\in I_n$. As the discussion in Section 4, we have the linear and nonlinear profile decomposition
\begin{align}\nonumber
e^{it\Delta }u_n(0,x) =  \sum_{j=1}^{k} v^j_n(t,x)+ e^{it\Delta} \omega^k_n,\; v^j_n(t,x) = e^{i(t -t^j_n)\Delta} \Tnj \snj \vp^j;\\\label{nlp-2}
\tilde{u}^{ k}_n\triangleq\sum\limits_{j=1}^k \ujn+e^{it \Delta} \omega^k_n,\;\; \ujn(t,x)=\Tnj\snj\Uij\left(\frac{t-\tnj}{(\hnj)^2}\right);\\\nonumber
\norm{\ujn(0)-\vnj(0)}_{H^1(\R^4)}\ra 0,\; \text{ as } n\ra \infty.
\end{align}

By Proposition 4.6, we have that there exists at least one profile $j_0$ such that
$\norm{\sigma^{j_0}_\infty U_\infty^{j_0}}_{ST^{j_0}(I^{j_0})}=\infty$.

Moreover, we claim that there is only one profile in the profile decomposition. In fact, this follows from the definition of $E^*_C$
and the fact that every solution of \eqref{001} in $\K^+$ with  energy less than $E^*_C$ has global finite Strichartz norm.

In summary, we have
$$E(u^1_{(n)}(0))\ra E^*_C,\;\;\; u^1_{(n)}\in \K^+, \;\;\;\norm{\sigma^1_\infty U^1_\infty}_{ST^1(I_1)}=\infty\text{ and }
\norm{\omega^1_n}_{\dot{H}^1(\R^4)}\ra 0.$$

If $h^1_n\ra 0$, then we have $\sigma^1_\infty U^1_\infty$ satisfies
$i\pa_tu+\Delta u+|u|^2u=0$ and
$$E_c( \sigma^1_\infty U^1_\infty)=E^*_C<m,\; \;\; \norm{\sigma^1_\infty U^1_\infty}_{\dot{H}^1}< \norm{W}_{\dot{H}^1},\;\;\;
 \norm{\sigma^1_\infty U^1_\infty}_{W_1(I_1)}=\infty, $$ which contradicts with the results in \cite{Dodson-2014}.

Hence we have $h^0_n\equiv 1$. Let  $u_c=U^1_\infty$, which is a solution of \eqref{001} and satisfies $E( u_c)=E^*_C<m$, which implies  $u_c\in \K^+$.
 \eqref{u-compact} follows from the linear profile decomposition and the fact that $\norm{\omega^1_n}_{\dot{H}^1(\R^4)}\ra 0.$

To prove \eqref{uc-condition}, we only need to prove $u_c$ is a global solution. If not, we can choose a sequence $t_n$ tends to
the finite boundary of $I_1$. Then we have $u_c(t+t_n)$ with finite mass and satisfying the condition \eqref{u-condition}. Hence,
 by the above discussion, we can find a sequence $(t'_n, x'_n)$ and a function $\phi\in H^1(\R^4)$, such that

\begin{equation}\label{uc-phi}
\norm{u_c(t_n)-e^{-it'_n\Delta}\phi(x-x'_n)}_{\dot{H}^1(\R^4)} \ra 0, \text{ as } n\ra \infty.
\end{equation}

For any $\ep>0$, by the  Strichartz estimates, we can find $\delta>0$, such that
$$\norm{ e^{it\Delta} \phi}_{W_1([t_n-t'_n-\delta,t_n-t'_n+\delta])}+\norm{\nabla e^{it\Delta}\phi }_{L^3_{t,x}
([t_n-t'_n-\delta,t_n-t'_n+\delta])}<\ep.$$
This and \eqref{uc-phi} imply that
$$\norm{ e^{it\Delta}u_c(t_n)}_{W_1([-\delta,\delta])}+\norm{\nabla e^{it\Delta} u_c(t_n)}_{{L^3_{t,x}([-\delta,\delta])}}<\ep ,$$
for sufficiently large $n$ (up to a subsequence).
Then by the Strichartz estimates and the H\"older inequality, we have
\begin{equation}\label{}\begin{split}
                           &  \norm{u_c(t)}_{W_1([t_n-\delta,t_n+\delta])}+ \norm{ \nabla  u_c(t) }_{L^3_{t,x}( [t_n-\delta,t_n+\delta])} \\
                          \les & \ep + \norm{\nabla \left[f_1(u_c(t))+f_2(u_c(t))\right]}_{L^{3/2}_{t,x}([t_n-\delta,t_n+\delta])}  \\
                    \les   & \; \ep \;+ \|u_c(t)\|_{W_1([t_n-\delta,t_n+\delta])} ^2 \norm{ \nabla u_c(t)}_{V_0([t_n-\delta,t_n+\delta])} \\
                           & + \|u_c(t)\|_{W_1([t_n-\delta,t_n+\delta])} ^\frac{4}{3} \norm{ \nabla u_c(t)}^{1/3}_{V_0([t_n-\delta,t_n+\delta])}
                           \norm{ \nabla  u_c(t) }^{2/3}_{L^2_{t,x} ([t_n-\delta,t_n+\delta])}.
                        \end{split}
\end{equation}
By the standard continuous method and the fact that $u_c\in \K^+$, we have that   $\norm{u_c(t)}_{W_1([t_n-\delta,t_n+\delta])}<+\infty.$

On  the other hand, by Sobolev embedding relation $\dot{H}^1(\R^4)\hookrightarrow L^4(\R^4),$ we have that
$\norm{u_c(t)}_{W_2([t_n-\delta,t_n+\delta])}<+\infty$. Hence $\norm{u_c}_{ST([t_n-\delta,t_n+\delta])}<+\infty$, which contradicts
with the fact that  $u_c$ is not a global solution.

By the local well-posedness theory, if we take $\ep$ small enough, the solution $u_c(t)$ exists on $[t_n-\delta,t_n+\delta]$, which
 contradicts with the choice of $t_n$. Hence we have $u_c$ is a global solution.
\end{proof}

\subsection{Compactness of the critical element}
In this subsection, for the critical elements as obtained above, we give out their compactness property  and  an important corollary.
\begin{prop}\label{prop:critical}
Let $u_c$ be a forward critical element, i.e.
\begin{equation}\label{u-forward}
\norm{u_c}_{ST([0,\infty))}=\infty .
\end{equation}
Then there exists $x(t):(0,\infty)\ra \R^4$ such that the set
$$\{u_c(t,x-x(t)): 0<t<\infty\} $$
is precompact in $\dot{H}^s$ for any $s\in(0,1]$.
\end{prop}
\begin{proof}
By the observation that the mass of $u_c$ conserves of time, we just need to prove the precompacness property in the space $\dot{H}^1(\R^4)$.

To do this, for any sequence $\{t_n\}\subset[0,\infty)$, we need
find a sequence $\{x_n\}\in \R^4$ such that the set
$\{u_c(t_n,x+x_n)\}$ is precompact in $\dot{H}^1(\R^4)$. If there
exists a subsequence $\{t_{n_k}\}$ of $\{t_n\}$ converges, the claim
follows from the time continuousness of $u_c$. If $t_n$ converges to
$\infty$, by profile decomposition and the discussion in Lemma
\ref{lem:5.1}, there exist a sequence
$(x_n,t'_n)\in\R^4\times[0,\infty)$ and a function $\vp\in H^1$ such
that
\begin{equation*}
\norm{u_c(t_n)-e^{-it'_n\Delta}\vp(x-x_n)}_{\dot{H}^1(\R^4)}\ra 0, \text{ as } n\ra \infty.
\end{equation*}
Hence we will be done if $t'_n$ converges to a constant. Thus we just need to preclude the two cases below:
\begin{enumerate}
\item[(1)] If $t'_n\ra -\infty$, by the Strichartz estimates, we have
 \begin{equation*}
\begin{split}
   & \norm{e^{it\Delta}u_c(t_n)}_{ST([0,\infty))}+\norm{\nabla e^{it\Delta}u_c(t_n)}_{L^3_{t,x}([0,\infty))} +\norm{|\nabla|^{1/2}
   e^{it\Delta}u_c(t_n)}_{L^3_{t,x}([0,\infty))} \\
   \les& \norm{\Japnb e^{it\Delta}\vp}_{S([-t'_n,\infty))} +\norm{u_c(t_n)-e^{-it'_n\Delta}\vp(x-x_n)}_{\dot{H}^1\cap\dot{H}^{1/2}(\R^4)}.
 \end{split}
 \end{equation*}
\end{enumerate}
Notice that the mass of $u_c$ and  $e^{-it'_n\Delta}\vp(x-x_n)$ is bounded, which together with interpolation implies $\dot{H}^{1/2}(\R^4)$
norm of their difference also tends to $0$.
Hence for any $\ep>0$, by Strichartz estimates
and the fractional chain rule, we have
\begin{equation*}
\begin{split}
     &\norm{u_c}_{L^6_{t,x}([t_n,\infty))}+\norm{u_c}_{L^4_{t,x}([t_n,\infty))}+\norm{\nabla u_c}_{L^3_{t,x}([t_n,\infty))}
     +\norm{|\nabla|^{1/2} u_c}_{L^3_{t,x}([t_n,\infty))}\\
     \les &  \ep+ \norm{u_c}^{2}_{L^6_{t,x}([t_n,\infty))}\left[\norm{\nabla u_c}_{L^3_{t,x}([t_n,\infty))}+ \norm{|\nabla|^{1/2}
     u_c}_{L^3_{t,x}([t_n,\infty))}\right] \\
          & +\norm{u_c}^{4/3}_{L^4_{t,x}([t_n,\infty))}\left[\norm{\nabla u_c}_{L^3_{t,x}([t_n,\infty))}+\norm{|\nabla|^{1/2}
          u_c}_{L^3_{t,x}([t_n,\infty))}\right],
   \end{split}
\end{equation*}
by taking $n$ sufficiently large.
By the standard continuous method, we can find that $\norm{u_c}_{ST([t_n,\infty))}<+\infty$, which contradicts with \eqref{u-forward}.
\begin{enumerate}
\item[(2)] If $t'_n\ra +\infty$, then we have \end{enumerate}
\begin{equation*}
\begin{split}
 &\norm{e^{it\Delta} u_c(t_n)}_{ST((-\infty,0))}+ \norm{\nabla e^{it\Delta}u_c(t_n)}_{L^3_{t,x}([0,\infty))} +\norm{|\nabla|^{1/2} e^{it\Delta}u_c(t_n)}_{L^3_{t,x}([0,\infty))}\\
 \les & \norm{\Japnb e^{it\Delta} \vp}_{S((-\infty, -t'_n))}+\norm{u_c(t_n)-e^{-it'_n\Delta}\vp(x-x_n)}_{\dot{H}^1\cap\dot{H}^{1/2}(\R^4)}\ra 0.
\end{split}
\end{equation*}
By a similar way as in (1), we have $\norm{u_c}_{ST(-\infty,t_n)}\ra 0$, which implies $u_c=0$ and contradicts with \eqref{u-forward}.

\end{proof}

By Proposition \ref{prop:critical} and the Arzela-Ascoli theorem, we
have
\begin{cor}\label{cor:compact}
Let $u_c(t,x)$ be a critical element, then there exists
$x(t):(0,\infty)\ra \R^4$ such that: for any $\eta>0$, there exists
$C(\eta)<\infty$ such that
$$
\int_{|x-x(t)|>C(\eta)} |\nabla u(t,x)|^2 dx+
\int_{|\xi|>C(\eta)}|\xi|^2|\hat{u}(t,\xi)|^2d\xi+\int_{|\xi|<\frac{1}{C(\eta)}}|\xi|^2|\hat{u}(t,\xi)|^2d\xi<\eta
$$
\end{cor}

\subsection{Extinction of the critical element }
In this part, we will finish the proof of  the main theorem. We will use the interaction Morawetz estimate to exclude the critical
 element in Proposition \ref{prop:critical}.
\begin{thm}
Suppose $u$ is a critical element in Proposition \ref{prop:critical}
satisfying that
$$u\in\K^+=\{u(t)\in H^1:E(u)<m,\;K_{\al,\beta}(u)\geq 0,\;(\al,\beta)\in \Omega \}. $$
Then we have $u\equiv0$.
\end{thm}
\begin{proof}
Let $\psi\in C^\infty_0(\R)$ be a radial function such that $\psi(x)=1$ when  $|x|\leq1$ and  $\psi(x)=0$ when $|x|>2$. Next, let \beq\label{phi} \phi(x-y)=\int \psi^2(x-s) \psi^2(y-s)ds. \eeq
Hence $\phi(x)$ is supported on $|x|\leq4$. For $1\leq R\leq R_0$
define the interaction Morawetz potential \beq\label{mrt}
M_R(t)=\int |u(t,y)|^2   \phi\Big(\frac{|x-y|}{R}\Big)   (x-y)_j
{\rm Im}[\overline{u}\partial_j u](t,x)dxdy \eeq and let
\beq\label{mt} M(t)=\int_{1\leq R\leq R_0} \frac{1}{R} M_R(t)dR.
\eeq

By the H\"{o}lder inequality and the Sobolev embedding theorem, we have
\beq\label{e-mt}
\sup_{t\in \R} |M_R(t)|\lesssim R^4.
\eeq

By a direct calculation, we have
\begin{align}\nonumber
&\frac{d}{dt}M_R(t)\\\label{m1}=&2\int |u|^2(t,y)
\phi\Big(\frac{|x-y|}{R}\Big)\Big[|\nabla
u|^2-|u|^4+\frac{4}{5}|u|^\frac{10}{3} \Big](t,x) dxdy\\\label{m2}
&-2\int{\rm Im}[\overline{u}\partial_j
u](t,y)\phi\Big(\frac{|x-y|}{R}\Big){\rm Im}[\overline{u}\partial_j
u](t,x)dxdy\\
\label{m3} &+2 \int
|u|^2(t,y)\phi'\Big(\frac{|x-y|}{R}\Big)\frac{(x-y)_j(x-y)_k}{|x-y|R}
{\rm Re}[(\partial_j\overline{u}\partial_ku)](t,x)dxdy\\\label{m4}
  &+\int |u|^2(t,y)\phi'\Big(\frac{|x-y|}{R}\Big)\frac{|x-y|}{R}\Big[|\nabla u|^2-\frac{1}{2}|u|^4+\frac{2}{5}|
  u|^\frac{10}{3}\Big](t,x)dxdy\\\label{m5}
 &-2\int
 {\rm Im}[\overline{u}\partial_k](t,y)\phi'\Big(\frac{|x-y|}{R}\Big)\frac{(x-y)_j(x-y)_k}{|x-y|R}{\rm Im}[\overline{u}\partial_ju](t,x)dxdy\\\label{m6}
  &-\frac{1}{2}\int |u|^2(t,y)\Delta\Big[4\phi(\frac{|x-y|}{R})+\phi'\Big(\frac{|x-y|}{R}\Big)\frac{|x-y|}{R}\Big] |u|^2(t,x)dxdy.
\end{align}

By Proposition \ref{prop:critical} and conservation of mass, we have
$u(t,x)\in L_t^\infty H^1(\R^4)$. Let $I=(0,K)\subseteq\R$ and
$1\leq R_0\leq K^\frac{1}{5}$.

First we consider \eqref{m3},  \eqref{m4} and \eqref{m5}. From the
definition of $\phi$, we have $$ \int_{1\leq R\leq R_0} \frac{1}{R}
\Big|\phi'\Big(\frac{|x-y|}{R}\Big)
\frac{(x-y)_j(x-y)_k}{|x-y|R}\Big| dR \les\int_{1\leq R\leq R_0}
\frac{1}{R}
 \Big|\phi'\Big(\frac{|x-y|}{R}\Big)\Big|   \frac{|x-y|}{R} dR  \les 1,
$$
and is supported on $|x-y|\les R_0$. Hence we have
\beq\label{e-m345}\Big| \int_I\int_{1\leq R\leq
R_0}\frac{1}{R}[\eqref{m3}+  \eqref{m4} + \eqref{m5}]dRdt\Big| \les K.
\eeq
 Next we consider \eqref{m6}. Notice that
\begin{align}\nonumber
     & \left|\int_{1\leq R\leq R_0} \frac{1}{R} \Delta\Big[4\phi\Big(\frac{|x-y|}{R}\Big)+\phi'\Big(\frac{|x-y|}{R}\Big) \frac{|x-y|}{R}\Big] dR \right|
      \\\nonumber
 =& \left|\int_{1\leq R\leq R_0} \frac{1}{R} \left[\phi'''\Big(\frac{|x-y|}{R}\Big) \frac{|x-y|}{R^3} +\phi''\Big(\frac{|x-y|}{R}\Big)\frac{6}{R^2}
 +\phi'\Big(\frac{|x-y|}{R}\Big)\frac{15}{R|x-y|}\right] dR\right|\\
       \les & \frac{1}{1+|x-y|^2}+\frac{1}{|x-y|+|x-y|^2}.
    \end{align}
This together with the Hardy-Littlewood-Sobolev inequality, the Sobolev theorem and $u\in L_t^{\infty}H_x^1$ implies that
\beq\label{e-m6}
\begin{split}
&\int_{I}\int_{1\leq R\leq R_0} \eqref{m6} \frac{1}{R}dRdt  \\
\les &\int_I \int|u|^2(t,x) \left( 1+\frac{1}{|x-y|^2} \right)|u|^2(t,y)dxdydt\\
 \les& K\left( \|u\|^4_{L^\infty_t L_x^2}+\|u\|^4_{L^\infty_t L_x^{\frac{8}{3}}}           \right)\\
 \les &K.
 \end{split}
\eeq

Now we will estimate \eqref{m1} and \eqref{m2}.
 By abuse of notations, we will write $\psi(x)=\psi(|x|)$, for $x\in\R^4.$
 Notice that for each $s$ and
$t$, there exists a $\xi(s,t)$ such that

\beq
\begin{split}
&\int \psi^2\Big(\frac{x}{R}-s\Big){\rm Im}[\overline{e^{ix\cdot\xi(s,t)}u} \nabla(e^{ix\cdot\xi(s,t)}u )](t,x) dx\\
=&\int \psi^2\Big(\frac{x}{R}-s\Big){\rm Im}[\overline{u} \nabla u](t,x) dx+\xi(s,t)\int\psi^2\Big(\frac{x}{R}-s\Big)|u|^2(t,x)dx \\
=&0.
\end{split}
\eeq
Moreover, for any $s,\;t$ and $\xi(s,t)$, the quantity
$$
\int\psi^2\Big(\frac{x}{R}-s\Big)\psi^2\Big(\frac{y}{R}-s\Big)\Big[|\nabla
u(t,x)|^2 |u(t,y)|^2- {\rm Im}[\overline{u}\nabla u](t,x) {\rm
Im}[\overline{u}\nabla u](t,y) \Big]dxdy
$$
is  invariant under the Galilean transformation $u(t,x)\mapsto e^{ix\cdot\xi(s,t)}u(t,x)$. Therefore, for any given $s,t$ we can choose
 $\xi(s,t)$ to eliminate the momentum squared terms.
Then, from integration by parts, we have
\begin{align}
     & \int\psi^2\Big(\frac{x}{R}-s\Big)\big[|\nabla  (e^{ix\cdot\xi(s,t)} u(t,x))|^2-|u(t,x)|^4\big ]dx \\
    =& \int     \Big|\nabla\Big(\psi\Big(\frac{x}{R}-s\Big)e^{ix\cdot\xi(s,t)} u(t,x)\Big)\Big|^2dx
    -     \int\psi^2\Big(\frac{x}{R}-s\Big)|u(t,x)|^4dx\label{improve1}       \\
     & + \int |u(t,x)|^2 \psi\Big(\frac{x}{R}-s\Big)\Delta_x\psi\Big(\frac{x}{R}-s\Big) dx .\label{improve2}
    \end{align}
Since $u\in \K^+$, we have $\|u\|_{L^\infty_t \dot{H}^1}\leq
(1-\bar{\delta})\|W\|_{\dot{H}^1}$ by Lemma \ref{lem:keylemma}.
Therefore, by the sharp Sobolev inequalities, we have \beq
\|u\|_{L^4(\R^4)} \leq \frac{\|\nabla
u\|_{L^2(\R^4)}}{\|W\|_{L^4(\R^4)}}\leq
\frac{1}{\|W\|_{L^4(\R^4)}}(1-\bar{\delta})\|W\|_{\dot{H}^1}\leq
(1-\bar{\delta})\|W\|_{L^4(\R^4)}. \eeq Let
$v=e^{ix\cdot\xi(s,t)}u$, then by H\"{o}lder inequalities we have
\beq\begin{split}
\eqref{improve1} \geq &\|\nabla(\psi v)\|^2_{L^2}-(1+\bar{\delta})\|\psi^2|u|^4\|_{L^1}+\bar{\delta}\|\psi^2|u|^4\|_{L^1}\\
 \geq & \|\nabla(\psi v)\|^2_{L^2}- (1+\bar{\delta})(1-\bar{\delta})^2\|\nabla(\psi v)\|^2_{L^2}+\bar{\delta}\|\psi^2|u|^4\|_{L^1}\\
 \geq &\bar{\delta}\|\psi^2|u|^4\|_{L^1}.
\end{split}
\eeq For \eqref{improve2}, noting that \beq
\int\psi\Big(\frac{x}{R}-s\Big)\Big|\Delta_x\psi\Big(\frac{x}{R}-s\Big)\Big|
\cdot\psi^2\Big(\frac{y}{R}-s\Big) ds \les
\psi\Big(\frac{x-y}{4R}\Big)\left(\frac{1}{R^2}+\frac{1}{R|x-y|}\right),
\eeq we have
\begin{align}\label{m12}
 &\Big|\int_I\int_{1\leq R\leq R_0} \frac{1}{R}\int\eqref{improve2}|u|^2(t,y)dydsdRdt\Big|
 \\\nonumber
 \les &\int_I\int\int_{1\leq R\leq R_0} \frac{1}{R}|u|^2(t,x) \left[\psi\Big(\frac{|x-y|}{4R}\Big)
 \Big(\frac{1}{R^2}+\frac{1}{R|x-y|}\Big)\right]|u|^2(t,y) \;dR\;dx\;dy\;dt\\
 \nonumber
 \les& \int_I\int |u|^2(t,x) \left(1+\frac{1}{|x-y|^2}\right)|u|^2(t,y)dxdydt\\\nonumber
 \les &K.
\end{align}
For $|x-y|\leq R_0^\frac{1}{2}$ and $R\geq R_0^\frac{1}{2}$, we have
\beq \int \psi^2\Big(\frac{x}{R}-s\Big)\psi^2\Big(\frac{y}{R}-s\Big)
ds \gtrsim 1. \eeq Thus, we have \beq\label{e-m1}
\begin{split}
&\int_I\int_{R_0^\frac{1}{2}\leq R\leq R_0} \frac{1}{R}\int |u|^2(t,y)  \psi^2\Big(\frac{x}{R}-s\Big)\psi^2\Big(\frac{y}{R}-s\Big)|u|^4(t,x)
dsdxdydRdt\\
\gtrsim & \int_I\int_{R_0^\frac{1}{2}\leq R\leq R_0} \frac{1}{R}\int_{|x-y|\leq R_0^\frac{1}{2}} \int |u|^2(t,y) |u|^4(t,x)  dxdydRdt\\
\gtrsim  & \ln{R_0} \int_I\int_{|x-y|\leq R_0^\frac{1}{2}} \int |u|^2(t,y) |u|^4(t,x)  dxdydt\\
\gtrsim & \ln{R_0} \int_I \left[\int_{|x-x(t)|\leq R_0^\frac{1}{4}}|u|^4(t,x)dx \cdot \int_{|y-x(t)|\leq R_0^\frac{1}{4}}|u|^2(t,y)dy \right]dt.
\end{split}
\eeq This together with \eqref{e-mt},\eqref{e-m345} and
\eqref{e-m6} implies \beq
\begin{split}
R_0^4 \gtrsim &\int_{1\leq R \leq R_0} \frac{1}{R} \sup_t |M_R((t))|  dR \\
\gtrsim & \left|\int_I \int_{1\leq R\leq R_0}\frac{1}{R}    \frac{d}{dt} M_R(t) dt dR  \right|                            \\
\gtrsim &\ln{R_0} \int_I \left[\int_{|x-x(t)|\leq R_0^\frac{1}{4}}|u|^4(t,x)dx \cdot \int_{|y-x(t)|\leq R_0^\frac{1}{4}}|u|^2(t,y)dy \right]dt- K.
\end{split}
\eeq
Since $1\leq R_0\leq K^\frac{1}{5}$, we have
\beq\label{000}
\ln{R_0}\int_I \left[\int_{|x-x(t)|\leq R_0^\frac{1}{4}}|u|^4(t,x)dx \cdot \int_{|y-x(t)|\leq R_0^\frac{1}{4}}|u|^2(t,y)dy \right]dt \les K.
\eeq
Furthermore, \eqref{000} implies that there exists a sequence $t_n\in R$ such that either

\beq\label{0001} \int_{|x-x(t_n)|\leq R_{0,n}^{\frac{1}{4}}}
|u|^4(t_n,x)dx\ra 0,\;R_{0,n}\ra \infty, \eeq or \beq\label{0002}
\int_{|y-x(t_n)|\leq R_{0,n}^{\frac{1}{4}}} |u|^2(t_n,y)dy\ra
0,\;R_{0,n}\ra \infty. \eeq By Corollary \ref{cor:compact} and the
Sobolev inequalities, \eqref{0001} and \eqref{0002} implies
$u\equiv0$.
\end{proof}

\textbf{Acknowledgments:}The authors thank the referee and the associated editor for their invaluable comments
and suggestions which helped improve the paper greatly. This work was supported in part by the National Natural Science Foundation of China under grants 11671047.

\begin{center}

\end{center}

\end{document}